\documentclass[12pt, letterpaper, twoside]{article}
\newcommand{\real}{{\mathbb R}}

\usepackage[legalpaper, letterpaper, margin=0.65in]{geometry}

\usepackage{comment}
\usepackage{amsmath,amsfonts,theorem,epsfig,amssymb}
\numberwithin{equation}{section}
\newtheorem{lemma}{Lemma}
\newtheorem{theorem}{Theorem}
\newtheorem{remark}{Remark}
\newtheorem{assumption}{Assumption}
\newtheorem{corollary}{Corollary}

\newcommand{\qed}{\mbox{\rule{.4em}{1.7ex}\hspace{.6em}}}
\newenvironment{proof}{{\bf Proof\ }}{\hspace*{.1em}\hfill\qed
\bigskip \noindent}
\usepackage{tikz}
\usepackage{tikz}
\usetikzlibrary{shapes, shadows, arrows}
\usetikzlibrary{positioning}

\usepackage{graphicx}
\usepackage[]{graphicx}
\usepackage{caption}
\usepackage{subcaption}
\usepackage{authblk}
\usepackage{fancyhdr}

\usepackage[utf8]{inputenc}
\usepackage[english]{babel}
\usepackage{biblatex}
\title{A generalized model of flocking with steering}

\author[1]{Guy A. Djokam }
\author[2]{ Muruhan Rathinam}
\affil[1]{Department of Mathematics and Statistics, University of Maryland Baltimore County}
\affil[2]{Department of Mathematics and Statistics, University of Maryland Baltimore County}
\date{}

\begin{document}

\maketitle

% REQUIRED
 
\begin{abstract}
We introduce and analyze a model for the dynamics of flocking and steering of a finite number of agents. In this model, each agent's acceleration consists of flocking and steering components. The flocking component is a generalization of many of the existing models and allows for the incorporation of many real world features such as acceleration bounds, partial masking effects and orientation bias. The steering component is also integral to capture real world phenomena. We provide rigorous sufficient conditions under which the agents flock and steer together. We also provide a formal singular perturbation study of the situation where flocking happens much faster than steering. 
 We end our work by providing some numerical simulations to illustrate our theoretical results. 
\end{abstract}
% REQUIRED
%\begin{keywords}
%flocking, steering, masking effect, orientation bias, acceleration bound.
%\end{keywords}

% REQUIRED
%\begin{AMS}
%%\end{AMS}
%%%%%%%%%%%%%%%%%%%%%%%%%%%%%%%%%%%%%%%%%%%%%%%%%%%%%%%%%%%%

\section{Introduction}
The emergence of phenomena such as flocking of birds, schooling of fish and swarming of bacteria have attracted considerable attention by mathematicians, scientists and engineers in the recent years. See \cite{ahn2012collision, cucker2010avoiding, cucker2007emergent,cucker2004modeling, motsch2011new, olfati2006flocking}, and references therein. Studying these phenomena not only help us understand the natural world, but also help us better engineer systems such as unmanned aerial vehicles. 
%The term {\em flocking} represents a phenomenon in which self-propelled agents may organize into an orderly motion by using only limited environmental information and simple rules. 
In \cite{vicsek1995novel}, Viscek and his team introduced a novel discrete time dynamics to investigate the emergence of self ordered motion. In Viscek's model, all agents have the same absolute velocity and at each step, they adjust their orientation based on their neighbors orientation. Inspired by this model, Cucker and Smale proposed the celebrated continuous time model  \cite{cucker2007emergent}, which led to many other subsequent studies. The Cucker-Smale (CS) model is:
for $i=1,\dots,N$ and $t \geq 0$
\begin{equation}   %\label{eqCS}  
\begin{aligned}
\frac{dx_i}{dt}&=v_i,  \\
\frac{dv_i}{dt}&= \frac{\alpha}{N}\sum_{j=1}^N a_{ij}(v_j-v_i),
\end{aligned}
\end{equation}
where $N$ is the number of agents, $x_i$ and $v_i$ are the position and velocity of agent $i$, and the {\em influence} $a_{ij}$ of agent $j$ on agent $i$ is assumed to be symmetric ($a_{ij}=a_{ji}$) and is a function of the Euclidean distance $\|x_i-x_j\|$ between $i$ and $j$, so that $a_{ij}=\phi(\|x_i-x_j\|)$.  
The function $\phi$ was chosen to be $\phi(r) = \frac{K}{(a^2 + r^2)^{\beta}}$, so that it was positive and non increasing. 

Cucker and Smale defined {\em flocking} by the condition that $$\sup_{t \geq 0}\|x_i(t)-x_j(t)\|<\infty$$ and that $$\lim_{t \to \infty} \|v_i(t)-v_j(t)\|=0$$ for every pair $(i,j)$ of agents.
The analysis of the CS model is based on the parameter $\beta$ and it is shown \cite{cucker2007emergent} that if $\beta < 1/2$, there is unconditional flocking and if $\beta \geq 1/2$ then flocking depends on initial conditions.  While the symmetric property of the influence functions led to ease of mathematical analysis, 
it is not realistic to assume symmetry. 

%the conservation of the first momentum and therefore, the asymptotic velocity is a combination of initial velocity. That is $v^{\infty} = \frac{1}{N}\sum_{j=1}^N v_j(0)$.  

Motivated by the CS model, many variants have been extensively studied in the literature. For instance, in \cite{cucker2010avoiding,park2010cucker} the authors propose models to address collision avoidance and in \cite{ha2010emergent} the authors study a modified CS model with nonlinear velocity couplings. A stochastic version of the CS model with multiplicative white noise is studied in \cite{ahn2010stochastic, ha2009emergence}. In \cite{shen2008cucker,aureli2010coordination,motsch2011new} authors study model with hierarchical leader. 
%Chun-Hsien Li and SUH-YUH YANG when further 
%In \cite{li2016flocking}, a discrete time CS model  is introduced where each agent has its own intrinsic dynamic under hierarchical leadership. This  intrinsic dynamic allowed the authors to show that the asymptotic velocity is not constant in time. The study of most of this model is based on the adjacent matrix of the graph form by the agents and the . 
An elegant analysis of flocking via the use of
a system of differential inequalities coupled with a Lyapunov function was introduced in
\cite{ha2009simple}.
In  \cite{motsch2011new}, Motsch and Tadmor present a more general model 
where the symmetry assumption on the influence functions is dropped. The Motsch and Tadmor (MT) model is given by
\begin{equation}\label{eqMT}
\begin{aligned}
\frac{dx_i}{dt} &= v_i, \\
\frac{dv_i}{dt} &= \alpha(\overline{v}_i-v_i),
\end{aligned}
\end{equation}
for $t \geq 0$, $i=1,\dots,N$, where $ \overline{v}_i = \sum_{j=1}^N a_{ij}v_j$ is a convex combination of the influences of all agents $j$ on agent $i$ so that
 $\sum_j a_{ij} =1$ and $a_{ij} \geq 0$. In this model, $\alpha>0$ is a
constant while $a_{ij}$ are taken to be some function of the pairwise
distances of the following form: 
\[
a_{ij}(x) = \frac{\phi(\|x_i-x_j\|)}{ \sum_{k} \phi(\|x_i-x_k\|)},
\]
where $\phi$ is a nonnegative function of distance. This form of $a_{ij}$ leads to lack of symmetry ($a_{ij} \neq a_{ji}$) and necessitated Motsch and Tadmor to introduce some new ideas into the analysis of flocking; in particular the concept of {\em maximal action} by a skew-symmetric matrix and the notion of an {\em active set}.

Our study is based on a finite number of agents where each agent follows a
similar rule though parameters appearing in these rules may vary from agent to
agent. The notion of the presence of leader agents is an important concept and
has been investigated in \cite{motsch2011new, shen2008cucker}.  It is
important to mention the development of continuum models which arise as
limiting models when the number of agents approaches infinity. These models
are based on partial differential equations that describe the evolution of the
density of the agents that formed the system. See
\cite{canizo2011well,haskovec2013flocking,motsch2011new} and reference
therein. It must be noted that flocking models usually are concerned with a
number of agents moving in the physical space and Newton's laws dictate that
such systems have a second order dynamics so that it is the accelerations of
agents that are usually controlled.  Models of first order self-organized
systems commonly arise in other applications such as opinion dynamics models
or flocking situations where one may reasonably assume that agents can
directly control their velocities.  See \cite{hendrickx2013convergence,stamoulas2018convergence} for instance.  

In this manuscript, we further generalize the MT model in ways that are inspired by the ability to account for acceleration bounds, masking effects as well as orientation bias. 
We endeavor to keep the model as general and flexible as possible while ensuring flocking behavior. 
Moreover, despite these generalizations, we believe that many real world phenomena may not be captured by a model that only incorporates flocking mechanisms without what we call {\em steering}. By steering, we mean additional acceleration by each agent which accounts for their 
individual responses to other external influences such as the need to compensate friction and gravity, pursuit of targets and evasion of danger.  

We note that the introduction of steering terms have been studied in the
literature, usually in conjunction with the symmetric CS type models 
\cite{caponigro2015sparse, bongini2014emergence, albi2016selective}. 
These researchers appear to be motivated from an engineering perspective, 
and are focused on the question of how to use the steering terms (controls)
to accomplish certain goals such as unconditional flocking, stabilization of
flocks etc. In contrast, our perspective is motivated more by biological
systems where, in addition to some built-in
urge to flock, each agent has its own whims in response to the external
world. Additionally, the non-symmetric flocking interactions in our model
makes the flavor of the analysis different. 

One important phenomenon observed in nature that is not captured by the CS, MT
as well as our model, is the {\em mill ring} where all the agents exhibit
a circular motion about a common axis of rotation with constant (in time)
angular velocities. In the literature mill ring as well as the {\em flock
  ring} formations have been studied \cite{d2006self, carrillo2010self,
  albi2014stability, carrillo2014nonlinear}. These models include a velocity
dependent acceleration term of the form $\alpha v_i - \beta \|v_i\|^2 v_i$ in
addition to position dependent potential forces.
%Our model without steering cannot exhibit a (non-stationary) mill ring

The paper is organized as follows. In Section \ref{sec:gen-flock}, we motivate
our generalized flocking model via the need for acceleration bounds, the
presence of masking effects and orientation bias. We introduce the open loop
and closed loop aspects of the flocking model. Once the flocking part of the
model is described, we show that in the presence of friction the velocities of
all agents asymptotically approach zero. We also show that our model (without
steering terms) does not exhibit a nontrivial mill ring phenomenon. These and other considerations motivate us to the introduction of the steering forces. We also briefly discuss existence and uniqueness of solutions. In Section \ref{sec-analysis-of-flocking}, we provide an analysis of our model and prove some sufficient conditions on flocking. Section \ref{flock-steer-pertubation-approach} investigates the leading order behavior of the  flocking and steering model via a formal singular perturbation approach when flocking is much faster than steering. Numerical simulations are provided in Section \ref{Numerical_examples} that illustrate our analysis. 

\section{The generalized flocking and steering model} \label{sec:gen-flock} 
We first discuss the generalization of the flocking model  and then include steering. We observe that the Motsch-Tadmor model has two aspects. First is the velocity alignment aspect which is given by: $\dot{v}_i =\alpha(\overline{v}_i -v_i)$ where $\alpha>0$ is a constant and $\overline{v}_i = \sum_{j=1}^N a_{ij} v_j$, is a (time dependent) convex combination of $v_1,\dots,v_N$. Regardless of the nature of this combination, in the velocity space, the acceleration of agent $i$ is always pointed towards a point in the convex hull of all the velocities. The second aspect of the model involves how $a_{ij}$ depend on the positions $x_1,\dots,x_N$. We note that throughout this paper $\|z\|$ stands for the Euclidean norm of a vector $z \in \real^d$.

 \subsection{Apriori acceleration bounds}\label{sec:acc-bnd}
 We start with the reasonable assumption that the magnitude of the acceleration $\|\dot{v}_i\|$ of any agent $i$ may not exceed a certain predetermined value, say $A>0$.  
It is readily observed that in the Motsch-Tadmor model of \eqref{eqMT},
the acceleration of agent $i$ is always given by $\alpha (\overline{v}_i -v_i)$ and since $\alpha>0$ is independent of $t$ and $i$, this does not readily allow for the condition $\alpha \|\overline{v}_i-v_i\| \leq A$ to be satisfied. Simply relaxing the model to allow for $\alpha$ to depend on $i$ and $t$, readily provides for the condition on acceleration bound to be 
\[
\alpha_i(t) \leq \frac{A}{\|\overline{v}_i(t)-v_i(t)\|},
\]
which can always be satisfied since agent $i$ chooses a time varying value for $\alpha_i(t)$. Thus, one may regard $\alpha_i(t)$ as a scalar control input from agent $i$. The only condition on each agent $i$ is that the agent accelerates in a direction parallel to $\overline{v}_i -v_i$ and pointing in the same sense so that $\alpha_i(t)>0$.
A simple feedback law that each agent $i$ can implement may take the form
\begin{equation}
\alpha_i(t) = \xi_i(\overline{v}_i(t)-v_i(t)),
\end{equation}
where $\xi_i:\real^{d} \to [0,\infty)$. Then the condition on acceleration bound becomes $\xi_i(u) \leq A/\|u\|$.
Motivated by this discussion, we state the following assumption. 
\begin{assumption} \label{ass-xi}
For $i=1,\dots,N$, the functions  $\xi_i:\real^{d} \to (0,\infty)$ 
are $C^1$ (continuously differentiable), strictly positive and there exists $A>0$ such that 
\begin{equation}
  \xi_i(u) \leq A /\|u\|, \text{ for } u \neq 0, \; i=1,\dots,N.
\end{equation}  
\end{assumption}
We note that the $C^1$ assumption helps ensure existence uniqueness 
of solutions. 
A simple example of $\xi_i$ is given by
\begin{equation}
\xi_i(u) = \frac{A}{\sqrt{a^2 + \|u\|^2}}, \quad i=1,\dots,N,
\end{equation}
where $a>0$ is some constant. 

\subsection{Masking effect and orientation bias}\label{sec:mask-orient}
In the CS model, the influence of agent $j$ on $i$ is given by  
the form $a_{ij}=\phi(\|x_i-x_j\|)$ whereas in the MT model it is given by 
\[
a_{ij}=\phi(\|x_i-x_j\|)/\sum_k \phi(\|x_i-x_k\|),
\] 
where $\phi:[0,\infty) \to [0,\infty)$. This form assumes that the influence of $j$ on $i$ is a function of all the pairwise distances. This specific form is not general enough to model masking effects. In order to explain this, we refer to Figure \ref{mask-effect}. In the position space, if a third agent $l$ is present in the line segment joining agents $i$ and $j$, then the influence of $j$ on $i$ may be lesser than if there were no agents present. This motivates a very general form of position dependence for $a_{ij}$. Additionally, the effect of agent $j$ on agent $i$ will depend on the orientation of the field of view of agent $i$. It is natural to  consider the orientation of agent $i$ as the unit vector $v_i/\|v_i\|$. However, this is undefined when $v_i=0$. To avoid singularities, we consider agent $i$'s orientation $u_i$ to be a $C^1$ function of $v_i$, so that $u_i=\sigma_i(v_i)$ where $\sigma_i:\real^d \to \bar{B}^d$ where $\bar{B}^d$ is the closed unit ball in $\real^d$. 
An example of $\sigma_i$ is given by
\[
\sigma_i(u)=  \frac{u}{\sqrt{\|u\|^2 + b_i^2}},
\]
where $b_i$ is a nonzero real number. 
These two observations suggest the following form for $a_{ij}$:
\begin{equation}\label{eq-phi}
a_{ij} = \phi_{ij}(x;\sigma_i(v_i)),
\end{equation}
where $x=(x_1,\dots,x_N)\in \real^{Nd}$, $\sigma_i:\real^d \to \bar{B}^d$ and
$\phi_{ij}:\real^{Nd} \times \real^d \to [0,\infty)$. We note that $\bar{B}^d$
  is the closed unit ball in $\real^d$. Thus the influence of agent $j$ on agent $i$ can be a nuanced function of the positions of all the agents as well as the velocity of agent $i$. We state our assumptions on $\phi_{ij}$. 
\begin{assumption}\label{ass-phi}
For $1 \leq i,j \leq N$,  $\phi_{ij}:\real^{Nd} \times \real^d \to (0,\infty)$ are $C^1$ and strictly positive. Moreover, $\phi_{ij}$ are shift invariant in position:
\begin{equation}\label{eq-phi-shift}
\phi_{ij}(x_1+y,x_2+y,\dots,x_N+y;u) = \phi_{ij}(x_1,x_2,\dots,x_N;u),
\end{equation}
$\forall x \in \real^{Nd}, \forall y \in \real^d, \forall u \in \bar{B}^d$.
Additionally, $\sigma_i:\real^d \to \bar{B}^d$ are $C^1$.  
\end{assumption}
We note that the shift invariance assumption is reasonable since the influence of agent $j$ on agent $i$ must only depend on the relative positions of all the agents, but not on their absolute positions. As before, the $C^1$ assumption helps ensure existence uniqueness results. The strict positivity assumptions on $\phi_{ij}$ are utilized in our flocking results and are a statement of lack of complete masking. That is, each agent has a nontrivial influence on every other agent regardless of the relative configuration. 

While our goal in this paper is to develop a general model, we mention that
an example of an influence function $\phi_{ij}$ that incorporates the masking
effect is given in Appendix \ref{sec:mask-example}.

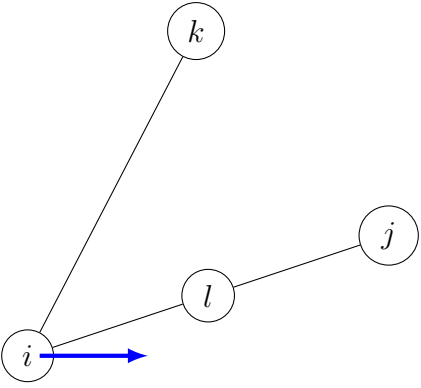
\begin{figure}[t]
  \begin{tikzpicture}
    [scale=.8,auto=left,every node/.style={}]
    \node (n1) at (0,0)[circle,draw]  {$i$};
    \node (n2) at (6,2)[circle,draw]  {$j$};
    \node (n3) at (2.8,5.4)[circle,draw] {$k$};
    \node (n4) at (3,1) [circle,draw] {$l$};
    \draw (n1)--(n3);
    \draw (n1)--(n4);
    \draw (n4)--(n2);
    \draw[->,ultra thick, blue,  arrows={-latex}]  (0.2,0) -- (2,0); %node[sloped, end ,above=-0.1cm] {$\mathsf{u_i}$};
 \end{tikzpicture}
\caption{\em Masking effect and orientation bias. The agents $j$ and $k$ are equidistant from agent $i$. Nevertheless, agent $l$ contributes to masking effect which diminishes agent $j$'s influence on agent $i$. On the other hand, agent $i$ is moving to the right and in agent $i$'s field of view agent $j$ is in a more prominent position than agent $k$, which diminishes agent $k$'s influence on agent $i$.}
\label{mask-effect}
\end{figure}

\subsection{The open loop and closed loop models}\label{sec:open-closed-flock}
It is instructive to consider our general model as forming two layers. The first layer, is the ``open loop'' model given by
\begin{equation}\label{eq-open-loop}
\begin{aligned}
\dot{x}_i &= v_i,\\
\dot{v}_i &= \alpha_i (\overline{v}_i-v_i),\\
\overline{v}_i &= \sum_{j=1}^N a_{ij} v_j,\\
a_{ij} &\geq 0 ,\;\; \sum_{j=1}^N a_{ij}=1,\;\; \alpha_i \geq 0\\
\end{aligned}
\end{equation}
where $\alpha_i$ and $a_{ij}$ are considered to be given functions of $t$,
which can be regarded as control inputs from agent $i$. The second layer of
our model specifies how $\alpha_i$ and $a_{ij}$ are chosen as functions of
positions and velocities, thus ``closing the loop''. The closed loop model thus contains the equations
\begin{equation}\label{eq-closed-loop}
\begin{aligned}
\dot{x}_i &= v_i, \\
\dot{v}_i &= \alpha_i (\overline{v}_i-v_i),\\
\overline{v}_i &= \sum_{j=1}^N \phi_{ij}(x;u_i) v_j,\\
u_i &= \sigma_i(v_i),\\
\alpha_i &= \xi_i(\overline{v}_i-v_i),
\end{aligned}
\end{equation} 
for $i=1,\dots,N$, where $\xi_i$ and $\phi_{ij}$ satisfy Assumptions
\ref{ass-xi} and \ref{ass-phi} respectively.  

\subsection{Inclusion of friction}\label{sec:friction}
In the real physical world, forces such as aerodynamic friction are present. We will consider a form of friction which is proportional to some power of the velocity. We have the following (open loop) system:
\begin{equation} \label{dynamicalsyst_friction}
  \begin{aligned}
   \dot{x}_i & = v_i,  \\
   \dot{v}_i &= \alpha_i(t)(\overline{v}_i - v_i) - c_i\|v_i\|^r v_i,
   \end{aligned}
\end{equation} 
for $i=1,\dots,N$, where $r \geq 0$. The following lemma shows very trivial asymptotic behavior.   
\begin{lemma}
Suppose $\{x_i(t),v_i(t)\}_{i=1}^N$ is a $C^1$ solution of system \eqref{dynamicalsyst_friction}. Then for each $i$ 
\[
   \lim_{t\to\infty}v_i(t) = 0. 
\]
\end{lemma}
\begin{remark}We note that Lemma \ref{lem-loclip-abscont} and Lemma 
\ref{lem-max-deriv} given in the appendix will be frequently used in the proofs of the results in this paper.
\end{remark}
\begin{proof}
  We define an energy of the system by $E = \max_{1\leq j\leq N} E_j$ where $E_j = \frac{1}{2}\|v_j\|^2$.
   Then by Lemmas \ref{lem-loclip-abscont} and
\ref{lem-max-deriv} $E(t)$ is absolutely continuous and $dE/dt(t) = dE_i/dt(t)$ for almost all $t$ where $i=i(t)$ is 
an index of the maximum. 
Thus, for almost all $t$, 
\[
 \begin{aligned}
 \frac{dE}{dt} &=  \langle v_i,\dot{v}_i \rangle
 			 = \langle v_i,\alpha_i(\bar{v}_i - v_i) - c_i\|v_i\|^r v_i \rangle 
 			 =-c_i\|v_i\|^{r+2} +\alpha_i \langle \bar{v}_i,v_j \rangle -\alpha_i\|v_i\|^2\\
 			 &= -c_i\|v_i\|^{r+2} +\alpha_i\sum_j a_{ij} \langle v_i,v_j \rangle -\alpha_i\|v_i\|^2\\
 			 &\leq-c_i\|v_i\|^{r+2} -\alpha_i\|v_i\|^2 +\alpha_i\|v_i\|\sum_ja_{ij}\|v_j\|
 			 \leq-c_i\|v_i\|^{r+2} \\
 \end{aligned}
\]
where we have used the Cauchy-Schwartz inequality and the fact that 
$\|v_i\| \geq \|v_j\|$ for all $j$. We also note that the index $i$ in general varies with $t$.  
Letting $ \underline{c} = \min_i c_i,$ we have
\begin{equation}\label{Energydis}
	\frac{dE(t)}{dt} \leq -\underline{c}\|v_i\|^{r+2} \leq - 2^{\frac{r}{2}+1} \, \underline{c} \, (E(t))^{\frac{r}{2}+1}.
\end{equation}
 Multiplying both side by $(E(t))^{-\frac{r}{2}-1}$, we have 
 \[
  \begin{aligned}
    (E(t))^{-\frac{r}{2}-1}\frac{dE(t)}{dt} &\leq - 2^{\frac{r}{2}+1} \, \underline{c}\\
   -\frac{2}{r} \frac{dE(t)^{-\frac{r}{2}}}{dt} &\leq - 2^{\frac{r}{2}+1} \, \underline{c}
  \end{aligned}    
 \]
 integrating the last inequality from $0$ to $t$ after some algebra manipulation, we have 
 \[
   (E(t))^{-\frac{r}{2}} - (E(0))^{-\frac{r}{2}} \geq  2^{\frac{r}{2}}r\,\underline{c}t
 \] 
 which implies
 \[
   E(t) \leq \frac{1}{((E(0))^{-\frac{r}{2}} +2^{\frac{r}{2}}r\,\underline{c}t )^{\frac{2}{r}}}.
 \]
Thus 
$E(t) \to 0$ as $t \to \infty$.
\end{proof}
Thus the addition of the friction shows that the asymptotic velocities go to
zero.
We note that the inclusion of friction into the CS model
in conjunction with a ``self propulsion'' term has been analysed in the literature
\cite{ha2010asymptotic}.

\subsection{Mill ring phenomena}
In nature, collective motions of living species can often exhibit milling
phenomena. That is a phenomenon in which agents rotate about a common axis of
rotation with a constant angular velocity. See \cite{albi2014stability} and
references therein for
instance. It is instructive to see if our
model can have a milling solution. In fact we show here that it is not
possible to have a mill ring solution for our model without the addition of
the steering terms.  

For simplicity, we will assume that agents are in $\real^3$ and that the axis
of rotation is the $z$-axis. We allow for different radii $R_i$, different
(constant) angular velocities $\omega_i$ and different $z$ coordinate
values $k_i$ for agents $i=1,\dots,N$.  
So we look for a solution of the following form. For each $i$, 
\begin{equation}\label{eq-ring-mill-soln}
\begin{aligned}
 x_i(t)  &= R_i\cos(\omega_i t +\theta_i)e_1 + R_i\sin(\omega_it +\theta_i) + k_i e_3,\\
 v_i(t) &= -R_i\,\omega_i\sin(\omega_it +\theta_i)e_1 + R_i\,\omega_i\cos(\omega_it +\theta_i)e_2,\\
%\dot{v}_i = & \ddot{x}_i = -R_i\,\omega_i^2\cos(\omega_it +\theta_i)e_1 - R_i\,\omega_i^2\sin(\omega_it +\theta_i)e_2\\
%  &= -\omega_i^2(x_i - k_i\,e_3),
  \end{aligned}
\end{equation}
where $e_1, e_2$ and $e_3$ are the standard basis (unit) vectors.

\begin{lemma}\label{lem-no-ring-mill}
Suppose \eqref{eq-ring-mill-soln} is a solution of the open loop model
\eqref{eq-open-loop} with $\alpha_i(t) \geq 0$. Then $v_i(t)=0$ for all $i$ and $t$. That is, the
only possible mill ring solution is the stationary mill ring. 
\end{lemma}
\begin{proof}
Our proof is mainly algebraic and the reasoning applies at each time $t$ and
hence we suppress showing the dependence on $t$ of $x_i, v_i$ etc.
It readily follows from \eqref{eq-ring-mill-soln} that $v_i$ and $\dot{v}_i$ are perpendicular for each $i$. This implies that \[
  \langle \dot{v}_i,v_i\rangle = \alpha_i \, \langle \bar{v}_i - v_i, v_i\rangle = 0,
  \]
  and hence either $\alpha_i=0$ or $\langle \bar{v}_i - v_i, v_i\rangle = 0$.
If $\alpha_i=0$ then $\dot{v}_i=0$. Since $\dot{v}_i=-\omega_i^2 (x_i - k_i
e_3)$, either $\omega_i=0$ or $x_i=k_i e_3$. In either case $v_i=0$.

Alternatively $\langle \bar{v}_i - v_i, v_i\rangle = 0$.
Hence
\begin{equation}\label{eq:v-bar-v-norm}
 \langle \bar{v}_i, v_i \rangle = \|v_i\|^2.
\end{equation}
We first show that
\begin{equation}
   \|v_i\| = \|v_j\|  \,\,\,\,\,\forall i,j \in \{1,\dots,N\}. \label{eq:all-velocity-are-equal-in-ring}
\end{equation} 

From \eqref{eq:v-bar-v-norm} and the Cauchy Schwartz inequality, we have 
\[
    \|v_i\|^2 = \sum_{j=1}^N a_{ij}\langle v_i,v_j\rangle
    \leq \sum_{j=1}^N a_{ij}\|v_i\|\|v_j\|,
    \]
and hence     
$\|v_i\| \leq \sum_{j=1}^Na_{ij}\|v_j\|$ for each $i$. 

Choose $i$ such that $\|v_i\| = \max\|v_j\| $.
Suppose that there exists $k$ such that $\|v_k\|<\|v_i\|$.  Then 
  \[
  \|v_i\| \leq \sum_{j=1}^Na_{ij}\|v_j\|\, < \sum_{j=1}^Na_{ij}\|v_i\| = \|v_i\|,
  \]
leading to a contradiction.
Thus $\|v_i\| = \|v_j\|$ for all $i,j \in \{1,\dots,N\}$.
 % \end{proof}
Next we show that
\begin{equation}
  \overline{v}_i = v_i \,\,\,\,\forall i \in \{1,\dots,N\}.
\end{equation}
%\begin{proof}
 From \eqref{eq:v-bar-v-norm} 
and the Cauchy Schwartz Inequality, we have that 
$\|\overline{v}_i\|\|v_i\| \geq \|v_i\|^2$ and hence 
$\|\overline{v}_i\| \geq \|v_i\| $. On the other hand 
\[
  \|\overline{v}_i\|\leq  \sum_j a_{ij}\|v_j\|  = 
  \sum_j\,a_{ij} \|v_i\| =
  \|v_i\|,
  \]
where we have used \eqref{eq:all-velocity-are-equal-in-ring}. 
Hence $\|\overline{v}_i\| = \|v_i\|$ for all $i$, and therefore, from
\eqref{eq:v-bar-v-norm} we conclude that $\overline{v}_i = v_i$ for each $i$. 
As a consequence, for each $i$, 
$\dot{v}_i = \alpha_i(\overline{v}_i-v_i) = 0$.
Noting that
$\dot{v}_i = \ddot{x}_i = -\omega_i^2(x_i - k_i\,e_3)$,
we conclude that for each $i$,  
\[
 \omega_i = 0\,\, \text{or}\,\, x_i = k_i e_3. 
\]
In either case, $v_i = 0$ for each $i$. Thus all the agents are stationary.  
\end{proof}

We remark that contrary to our earlier stipulation that $\alpha_i(t)>0$ (which
was motivated by the need to guarantee flocking), in
Lemma \ref{lem-no-ring-mill} we allowed for the possibility $\alpha_i(t)=0$. 

\subsection{Steering}\label{sec:steer}

The previous two subsections illustrate certain shortcomings of the open
loop model \eqref{eq-open-loop} which primarily focuses on velocity
alignment. One is that the inclusion of friction
leads to unrealistic behavior without a term to compensate for it.
%This observation and other real world considerations show that in
%addition to pure flocking, there should be a ``steering'' component to each
%agent's acceleration.
Many researchers have incorporated a ``self-propulsion'' acceleration term
which is proportional to the agent's velocity ($k v_i$) to compensate friction. See for instance 
\cite{ha2010asymptotic, d2006self, carrillo2010self}.

The second shortcoming we observed is that \eqref{eq-open-loop} does not support a nontrivial
mill ring solution. It must be noted that flocking
models exhibiting mill ring phenomena have been studied in the literature.
See for instance \cite{albi2014stability, d2006self, carrillo2010self,
  carrillo2014nonlinear} and references therein. It must be noted that these
authors consider models that have attraction and/or repulsive forces
via a potential that depends on relative positions in conjunction with a
velocity dependent acceleration of the form $\alpha v_i - \beta \|v_i\|^2
v_i$. Our model as well as the
CS and MT models do not share this feature. 

Our approach to capture rich behavior in the flocking model is to introduce
a ``steering'' component to each agent's acceleration. We feel that these
steering terms make intuitive sense. In reality a group of agents may want to
follow a desired trajectory in addition to staying together as a flock.
%This necessitates an extra ``steering'' term.
Thus, each agent $i$ may have an extra acceleration $\beta_i(t)$ which contributes to steering. 
This steering term can also act to cancel other external forces such as
friction and gravity. We interpret $\beta_i(t)$ in the following as the
steering component in excess of friction and gravity.

This leads to the system
\begin{equation}\label{eq-flock-steer-open}
\begin{aligned}
\dot{x}_i &= v_i,\\
\dot{v}_i &= \alpha_i (\overline{v}_i-v_i) + \beta_i,\\
\overline{v}_i &= \sum_{j=1}^N a_{ij} v_j,\\
a_{ij} &\geq 0, \;\; \sum_{j=1}^N a_{ij}=1,\;\; \alpha_i \geq 0\\
\end{aligned}
\end{equation}
for the open loop with steering and
\begin{equation}\label{eq-flock-steer-closed}
\begin{aligned}
\dot{x}_i &= v_i, \\
\dot{v}_i &= \alpha_i (\overline{v}_i-v_i) + \beta_i,\\
\overline{v}_i &= \sum_{j=1}^N \phi_{ij}(x;u_i) v_j,\\
u_i &= \sigma_i(v_i),\\
\alpha_i &= \xi_i(\overline{v}_i-v_i),
\end{aligned}
\end{equation} 
for the closed loop with steering. 
\begin{assumption}\label{ass-beta}
The steering functions $\beta_i:[0,\infty) \to \real^d$ for $i=1,\dots,N$ are continuous.
\end{assumption}

We note that both the open loop \eqref{eq-flock-steer-open} and the closed
loop \eqref{eq-flock-steer-closed} are {\em differentially flat}
\cite{levine2011necessary, van1998differential} and hence any given
sufficiently smooth trajectory
for the positions $x(t)=(x_1(t),\dots,x_N(t))$ is feasible. This is easy to
see as given $x_i(t)$ (for $i=1,\dots,N$) one may readily obtain $v_i(t)$ and $\beta_i(t)$
from the equations. Thus the mill ring phenomenon is certainly possible. 

Finally, we observe that steering terms in conjunction with CS type models have been
introduced and studied from a control theoretic perspective in
\cite{caponigro2015sparse, bongini2014emergence, albi2016selective}.

\subsection{Existence and uniqueness}\label{sec-exist-unique}
We briefly discuss existence and uniqueness of solutions of the open loop and closed loop models \eqref{eq-flock-steer-open} and \eqref{eq-flock-steer-closed}. The open loop model is linear and non-autonomous and hence it is adequate to assume that $\alpha_i(t)$, $a_{ij}(t)$ and $\beta_i(t)$ are all continuous in time. The closed loop model is of the form
\[
\dot{z} = F(z) + \beta(t)
\]
where $z=(x_1,\dots,x_N,v_1,\dots,v_N) \in \real^{2Nd}$ and $F$ is $C^1$ by our assumptions on $\phi_{ij}$ and $\xi_i$. Again if we assume $\beta_i(t)$ to be continuous in $t$ then for any given initial condition for $z(0)$, we are assured of a unique solution in an open maximal interval of time containing $0$. 

In order to discuss flocking behavior, it is important to ensure that the forward maximal interval of existence is $[0,\infty)$. When the steering is open-loop, with Assumption \ref{ass-beta}, it is shown in Lemma \ref{lem-max-interval} that the forward maximal interval is infinite. When steering is considered to be closed-loop, that is some function of position and velocity, then a different analysis is needed.

\section{Analysis of flocking} \label{sec-analysis-of-flocking}
 
\subsection{Mathematical preliminaries} \label{sec:math-preleminaries}
First we define some relevant concepts and state some useful lemmas. Given the positions $x_i(t)$ and velocities $v_i(t)$ (where $i=1,\dots,N$) of agents, we denote by $d_X(t)$ and $d_V (t)$ the diameters in position and velocity spaces $\real^{Nd}$:
\begin{equation}
\begin{aligned}
        d_X(t)&=\max_{i,j} \|x_j(t) - x_i(t)\|,\\
        d_V (t)&=\max_{i,j} \|v_j(t) - v_i(t)\|.
\end{aligned}
\end{equation} 
The system $\{x_i(t),v_i(t)\} $ $i=1,...,N $ is said  to {\em converge to a flock}, if the following two conditions hold:
\begin{equation}\label{eq-flock}
        \sup_{t\geq 0}d_X(t)< \infty, \quad
        \lim_{t\to \infty}d_V(t)= 0.
\end{equation}
We define $d_\beta(t)$, the diameter in the ``steering space'' by
\begin{equation}\label{eq-dbeta-def}
d_\beta(t) = \max_{i,j} \|\beta_j(t) - \beta_i(t)\|.
\end{equation}
The flocking analysis in this paper uses the notion of {\em active sets} developed in \cite{motsch2011new}. Recall that $a_{ij}(t)$ denotes the influence of agent $j$ on agent $i$ at time $t$ and that $a_{ij}(t)\geq 0$  and $\sum_{j=1}^N a_{ij}(t) =1$. Given $\theta > 0$, it is instructive to consider the set of all agents who influence a given agent $1 \leq p \leq N$ by an amount greater than or equal to $\theta$. This is known as the {\em active set} $\Lambda_p(\theta)$ for agent $p$: 
\begin{equation}\label{eq-Lambda-p}
  \Lambda_p(\theta) = \{j \, | \, a_{pj} \geq \theta\}.
\end{equation}
For a pair of agents $p$ and $q$, the common active set $\Lambda_{pq}(\theta)$ 
is simply the intersection $\Lambda_p(\theta) \cap \Lambda_q(\theta)$. 
The {\em global active set} $\Lambda(\theta)$ is the intersection of all the
active sets:
\begin{equation}\label{eq-Lambda}
\Lambda(\theta) = \bigcap_p \Lambda_p(\theta).
\end{equation}
The following lemma from \cite{motsch2011new} is critical.
\begin{lemma}\cite{motsch2011new}\label{lem-antisym}
Let $S$ be and antisymmetric matrix, $S_{ij} = -S_{ji} $ with $| S_{ij}| \leq M$.
Let  $ u,w \in \real^n$ be two given vectors with positive entries, $u_i ,
w_i\geq 0$ and let $ \overline{U}, \overline{W}$ denoted their respective
sums, $ \overline{U}=\sum_{i} u_i $ and $ \overline{W}=\sum_jw_j$. 
Fix  $\theta>0 $ and let $\lambda(\theta)  $ denoted the number of ``active
entries'' of $u$ and $ w$ at the level $\theta$ in the sense that 
 $$ \lambda(\theta)=|\Lambda(\theta)|,$$  $$ \Lambda(\theta) = \{ j \, | \,u_j\geq \theta \overline{U}  \text{ and } w_j \geq \theta\overline{W}\}. $$ 
Then for every $\theta >0 $, we have $$|\langle Su,w\rangle|\leq M\overline{U}\overline{W}(1-\lambda^2(\theta)\theta^2).$$
\end{lemma}
For our analysis, in addition to Lemma \ref{lem-antisym}, we need the following simple lemma about the convex hull of a finite set of points in $\real^d$.  
\begin{lemma}\label{lem-hyperplane}
 Let $ \{v_i\}_{i=1}^N$ be a set of vectors in $\real^d$ and let $\Omega$ be 
their convex hull. If $v_p$ and $v_q$ delimit the diameter of the convex hull, 
(that is $v_p$ and $v_q$ are furthest apart), then for each $v \in \Omega$ 
\[ 
\langle v_p -v_q,v - v_q \rangle \geq 0.
\]   
\end{lemma}
 \begin{proof}
Let the diameter of $\Omega$ equal $\|v_p-v_q\|$.
We first show that 
\[
\langle v_p -v_q,v_i - v_q \rangle \geq 0 \quad \forall i.
\]
Let $ H_{pq} $ be the hyperplane passing through $ v_q$ and is perpendicular 
to $ v_p - v_q$. (See Figure \ref{fig-hyperplane}). Suppose there is some $i$ such that $\langle v_p -v_q, v_i - v_q\rangle < 0$. This shows that $v_i$ and $v_p$ will be on opposite sides of the hyper plane $H_{pq}$, implying that $\|v_i - v_p\| > \|v_q-v_p\|$, a  contradiction. 
Given any $v \in \Omega$,  there exist $a_i \geq 0$ for $i=1,\dots,n$ 
such that $\sum_{i=1}^N a_i=1$ and $v = \sum_{j=1}^N a_i v_i$. Hence
\[
\langle v_p -v_q,v - v_q \rangle = \langle v_p -v_q, \sum_{i=1}^N a_i (v_i-v_q)
\rangle \geq 0. 
\]
 \end{proof}

 \begin{center}
 \begin{figure}
\centering
\begin{tikzpicture}
  [scale=.8,auto=left,every node/.style={}]
  \node (n4) at (0,6)[circle,draw]  {$v_p$};
  \node (n5) at (2,7)[circle,draw]  {};
  \node (n1) at (4,5)[circle,draw] {$v_q$};
  \node (n2) at (2,4) [circle,draw] {};
  \node (n3) at (0,4) [circle,draw] {};
  \node (n6) at (6.5,7)[circle,draw] {$v_j$}; 
  \node (n7) at (3,8){.};
  \node (n8) at (5,9)  {.};
  \node (n9) at (2,1) {.};
  \node (n10) at (4,2) {$H_{pq}$};
  \foreach \from/\to in {n4/n5,n5/n1,n1/n2,n2/n3,n3/n4,n4/n6}
    \draw (\from) -- (\to);
    \foreach \from/\to in {n7/n8,n7/n9,n9/n10,n10/n8,n4/n1}
    \draw [thin, dash dot] (\from) -- (\to);
\end{tikzpicture}
\caption{\em Illustration of the lemma}
\label{fig-hyperplane}
\end{figure}
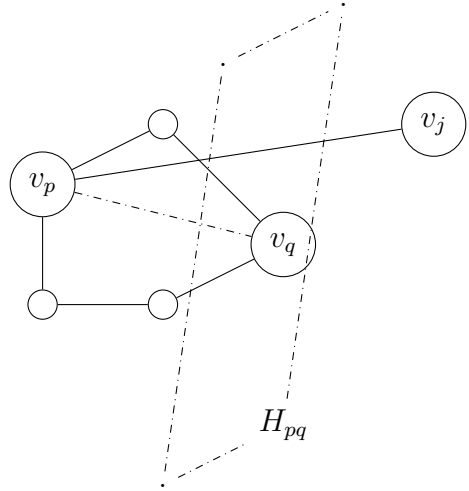
\end{center}

\subsection{Analysis of the open loop}
We shall suppose that Assumption \ref{ass-beta} holds. 
\begin{theorem} \label{thm1}
Let $(x(t),v(t)) \in \real^{Nd} \times \real^{Nd}$ be a $C^1$ solution of the open loop \eqref{eq-flock-steer-open}. At time $t$, let $d_V(t) = \|v_p(t)-v_q(t)\|$. 
Fix an arbitrary $\theta > 0$ and let $\lambda_{pq}(\theta)$ be the number of agents in the common active set $\Lambda_{pq}(\theta)$ associated with the influence matrix $a_{ij}(t)$ of the system. Let $\alpha_0(t) = \min_{i}\alpha_i(t)$. Then for almost all $t$, the diameters of the system, $d_X(t)$, $d_V(t)$ and $d_\beta(t)$ satisfy :
\begin{equation} %\label{thm1}
\begin{aligned}
  \frac{d}{dt}d_X(t)&\leq d_V(t)\\ %\label{thm1a}
  \frac{d}{dt}d_V(t)&\leq -\alpha_0 \lambda_{pq}^2(\theta) \theta^2 d_V(t) + d_\beta(t). \label{thm1b}
\end{aligned}
\end{equation} 
\end{theorem}

\begin{remark}
In Theorem \eqref{thm1}, we note that $p$ and $q$ are functions of $t$, and so is $\lambda_{pq}(\theta)$. 
This theorem is a generalization of Theorem 3.4 of \cite{motsch2011new} where $\alpha_i(t)$ was independent of $i$ and $t$. One needs Lemma \ref{lem-hyperplane} to handle the extra terms that appear in our analysis.
\end{remark}

\begin{proof}
By Lemma \ref{lem-loclip-abscont} $d_X(t)$ is absolutely continuous.
We choose $i=i(t)$ and $j=j(t)$ such that 
$d_X(t) = \|x_i(t)-x_j(t)\|$ for all $t$. 
Using Lemma \ref{lem-max-deriv}  we obtain
\[
\left|\frac{d}{dt} (d_X(t))^2 \right| =  \left|\frac{d}{dt}\|x_i - x_j\|^2\right|  = 2|\langle x_i - x_j,v_i -
 v_j\rangle|.
\]
Hence
\[
\left|2\|x_i -x_j\|\frac{d}{dt}\|x_i - x_j\|  \right| = 2|\langle x_i - x_j,v_i - v_j\rangle| \leq 2\|x_i -x_j\|\|v_i -v_j\|.
\]
This yields that for almost all $t$
\[
  \frac{d}{dt}d_X(t) \leq \|v_i - v_j\| \leq d_V(t).
\]
Note that if for some $t>0$, $d_X(t)=0$ and $d_X(t)$ is differentiable, then $\frac{d}{dt}d_X(t)=0$. 

For the second inequality, we again proceed by using Lemmas \ref{lem-loclip-abscont} and \ref{lem-max-deriv}. Let $p=p(t)$ and $q=q(t)$ be such that  $d_V(t) = \|v_p-v_q\|$ for all $t$. Then (for almost all $t$)
\[
\begin{aligned}
\frac{d}{dt}(d_V(t))^2 &= \frac{d}{dt}(\|v_p -v_q\|^2) = 2\langle v_p - v_q , \dot{v}_p - \dot{v}_q\rangle\\
& = 2\langle v_p -v_q,\alpha_p(\overline{v}_p - v_p) - \alpha_q(\overline{v}_q
- v_q)\rangle + 2\langle v_p - v_q,\beta_p - \beta_q\rangle\\
&=2\alpha_p\langle v_p -v_q,\overline{v}_p - v_p\rangle - 2\alpha_q\langle v_p -v_q,\overline{v}_q - v_q\rangle + 2\langle v_p - v_q,\beta_p - \beta_q\rangle.
\end{aligned} 
\]
We proceed by assuming WLOG that $ \alpha_p \leq \alpha_q$ and write
\[
\begin{aligned}
\frac{d}{dt}(\|v_p -v_q\|^2) &=2\alpha_p \langle v_p -v_q,\overline{v}_p - \overline{v}_q\rangle - 2\alpha_p\|v_p - v_q\|^2\\
& -2(\alpha_q -\alpha_p)\langle v_p -v_q,\overline{v}_q - v_q\rangle + 2\langle v_p - v_q,\beta_p - \beta_q\rangle.
\end{aligned} 
\]
Using Lemma \eqref{lem-hyperplane}, Cauchy-Schwartz inequality and the fact that $\|\beta_p-\beta_q\| \leq d_\beta$, we have:
\[
\begin{aligned}
\frac{d}{dt}(\|v_p -v_q\|^2)
&\leq 2\alpha_p \langle v_p -v_q,\overline{v}_p - \overline{v}_q\rangle - 2\alpha_p\|v_p - v_q\|^2 + 2\|v_p - v_q\| d_{\beta}. 
\end{aligned} 
\]
Moreover
\[
 \begin{aligned}
 	\overline{v}_p - \overline{v}_q &= \sum_{j=1}^N a_{pj}v_j - \overline{v}_q\ = \sum_{j=1}^N a_{pj}(v_j - \overline{v}_q)\\
 &= \sum_{j=1}^N a_{pj}(v_j - \sum_{i=1}^N a_{qi}v_i)
 	                                =\sum_{i,j}^N a_{pj}a_{qi}(v_j - v_i).\\
 	\end{aligned}
\]
Hence
\[
\begin{aligned}
\frac{d}{dt}(\|v_p -v_q\|^2)
&\leq 2\alpha_p\sum_{i,j}^N a_{pj}a_{qi}\langle v_p -v_q,v_j - v_i\rangle - 2\alpha_p\|v_p - v_q\|^2 + 2 d_\beta \|v_p-v_q\|. 
\end{aligned} 
\]
Now we use Lemma \eqref{lem-antisym} with $u_i = a_{pi}$, $w_i = a_{qi}$, and the anti-symmetric matrix 
\[
S_{ij} = \langle v_p - v_q,v_i - v_j\rangle.
\]
Since $|S_{ij}| \leq d_V^2$, we have
\[
\left|\sum_{i,j}^N a_{pj}a_{qi}\langle v_p -v_q,v_j - v_i\rangle \right| \leq d_V^2(1 - \lambda_{pq}^2(\theta)\theta^2).
\]
Therefore we have
 \[
  \frac{d}{dt}(\|v_p -v_q\|^2)
     \leq 2\alpha_p d_V^2(1 - \lambda_{pq}^2(\theta)\theta^2) - 2\alpha_p\|v_p - v_q\|^2 + 2 d_\beta \|v_p-v_q\|.
 \]
Noting that $v_p$ and $v_q$ are such that $\|v_p(t) -v_q(t)\| = d_V(t)$ 
and that $\alpha_p(t) \geq \alpha_0(t)$ by definition, we have
 \[
 \frac{d}{dt}(d_V(t)^2)
     \leq - 2\alpha_0 d_V^2 \lambda_{pq}^2(\theta)\theta^2 + 2 d_\beta d_V.
\]
An argument similar to the one used in deriving the first inequality proves
\eqref{thm1}.
\end{proof} 
The following corollary is immediate.
\begin{corollary}\label{corr-thm1}
If $\lambda(\theta)$ is the number of elements in the global active 
set $\Lambda(\theta)$ and if $\underline{\alpha}$ denotes 
the infimum of $\alpha_i(t)$ over $i$ and $t \geq 0$ then 
\begin{subequations}
\begin{align}
	\frac{d}{dt}d_X(t)&\leq d_V(t), \label{corr1a} \\
	\frac{d}{dt}d_V(t)&\leq -\underline{\alpha} \lambda(\theta)\theta^2d_V(t) + d_\beta(t).\label{corr1b}
\end{align}      
\end{subequations} 
\end{corollary}

\subsection{Analysis of the closed loop}
We shall suppose that Assumptions \ref{ass-xi}, \ref{ass-phi} and \ref{ass-beta} hold. 
As observed in Section \ref{sec-exist-unique} these Assumptions guarantee existence and uniqueness of a solution to the closed-loop equations on the time interval $[0,\infty)$. Moreover, this solution is $C^1$ in time $t$. We note that, if all steering terms $\beta_i$ are equal for all $t$, then $d_\beta(t)=0$ and the system of inequalities given by \eqref{corr1a} and \eqref{corr1b} show that the diameter $d_V$ is decreasing in time. Even in this case, in order to show flocking, one needs stronger inequalities. To that end, we shall modify the ideas from Ha et al \cite{ha2009simple} and also from Motsch and Tadmor \cite{motsch2011new} in order to prove the flocking results. 
We define the function $\psi:[0,\infty) \to (0,\infty)$ by 
\begin{equation}\label{eq-psi}
\psi(r) = \min_{1 \leq i,j \leq N} \min \{ \phi_{ij}(x;u) \, | \, \|x_l-x_k\| \leq r, \; u \in \bar{B}^d \text{ and }\; 1\leq l,k\leq N\}. 
\end{equation}
In order to see that the minimum exists, we observe that by shift invariance 
(Assumption \ref{ass-phi}), 
\[
\begin{aligned}
\{ \phi_{ij}(x;u) \, &| \, \|x_l-x_k\| \leq r, \; u \in \bar{B}^d\} \\
&= \{ \phi_{ij}(x;u) \, | \, x_1=0, \|x_l-x_k\| \leq r, \; u \in \bar{B}^d\}
\end{aligned}
\]
and that 
\[
\{ (x,u) \in \real^{Nd}\times \bar{B}^d \, | \, x_1=0, \|x_l-x_k\| \leq r, \; u
\in \bar{B}^d\}
\]
is a compact set and that $\phi_{ij}$ are continuous. Since $\phi_{ij}$ are strictly positive by Assumption \ref{ass-phi}, it follows that $\psi$ is strictly positive.  Moreover, it is also clear that $\psi$ is a decreasing (non-increasing) function. Since $\psi$ is decreasing, it is also positive and measurable, and hence $\int_0^{r_0} \psi(r) dr < \infty$ and $\int_0^\infty \psi(r) dr \leq \infty$ are well-defined.  
\begin{lemma}\label{lem-alpha-bar}
Let $\underline{\alpha}$ be the infimum of $\alpha_i(t)$ over $i$ and $t \geq 0$.
Suppose that
\[
\int_0^\infty d_\beta(t) \, dt < \infty.
\]
Then $\underline{\alpha}>0$.
\end{lemma}
\begin{proof}
Let $M$ be defined by
\[
M = d_V(0) + \int_0^\infty d_\beta(t) \, dt.
\]
Then from \eqref{corr1b} it follows that $d_V(t) \leq M$ for all $t \geq 0$.
Hence, for all $t \geq 0$ and for all $1 \leq i \leq N$, $\|\overline{v}_i(t)-v_i(t)\| \leq d_V(t) \leq M$,
where we have used the fact that $\overline{v}_i$ is in the convex hull 
of all velocities $v_j$.
Now 
\[
\begin{aligned}
\underline{\alpha} &= \inf\{\; \xi_i(\overline{v}_i(t)-v_i(t)) \; | \; t \geq 0, \; 1
\leq i \leq N \},\\
&\geq \min\{\; \xi_i(u) \; | \; 0 \leq \|u\| \leq M, \; 1 \leq i \leq N\} >0,
\end{aligned}
\]
where we have used the fact that $\xi_i$ is continuous by Assumption \ref{ass-xi}.
\end{proof}

\begin{theorem}\label{thm-closed-flock}
Consider the closed loop system \eqref{eq-flock-steer-closed}. Suppose $\psi$ is
defined by \eqref{eq-psi} and that 
\[
\int_0^\infty d_\beta(t) \, dt < \infty \quad \text{and} \quad \lim_{t\to \infty} d_{\beta}(t) = 0.
\]
Further suppose that the initial diameters satisfy
\begin{equation}\label{eq-dV0-bound}
d_V(0) + \int_0^\infty d_\beta(t) \, dt < \underline{\alpha} N^2 \int_{d_X(0)}^{\infty} \psi(s)ds.
\end{equation}
Then the solution $(x(t),v(t))$ flocks. In particular, 
if $\int^\infty \psi(s)\,ds = \infty$, then the  condition on initial diameters is always satisfied.
\end{theorem}
\begin{proof}
At any given time, by choosing $\theta(t) = \sqrt{\psi(d_X(t))}$, one readily obtains that the number of elements in the global active set is $N$, and hence the inequality 
\[
\frac{d}{dt} d_V(t) \leq -\underline{\alpha} N^2 \psi(d_X(t)) d_V(t) + d_\beta(t).
\]
 We define the energy functional
$\mathcal{E}: \real^{Nd} \times \real^{Nd} \to \real$ by
\begin{equation}
\mathcal{E}(d_X(t),d_V(t)) = d_V(t) + \underline{\alpha}N^2\int_0^{d_X(t)}\psi(s)ds.
\end{equation}
The time derivative of the energy functional satisfies
\[
\dot{\mathcal{E}} = \dot{d}_V + \underline{\alpha}N^2d_V\psi(d_X) \leq d_{\beta}.
\]
Hence
\[
 \mathcal{E}(d_V(t),d_X(t)) - \mathcal{E}(d_V(0),d_X(0)) \leq \int_0^t d_{\beta}(s)ds,
\]
which implies
\[
 d_V(t) - d_V(0) \leq -\underline{\alpha}N^2\int_0^{d_X(t)}\psi(s)ds +\underline{\alpha}N^2\int_0^{d_X(0)}\psi(s)ds  +\int_0^t d_{\beta}(s)ds.
\]
We deduce that
\begin{equation}\label{equation-diameter}
 d_V(t) - d_V(0) \leq \underline{\alpha}N^2\int_{d_X(t)}^{d_X(0)}\psi(s)ds + \int_0^t d_{\beta}(s)ds.
 \end{equation}
By the assumption \eqref{eq-dV0-bound}, there exists $d_*$ (independent of $t$) such that
\begin{equation}
 \int_0^{\infty}d_{\beta}(t) \, dt + d_V(0)\leq\underline{\alpha}N^2\int_{d_X(0)}^{d_*}\psi(s)ds. 
\end{equation}
Replacing this inequality in \eqref{equation-diameter}, we obtain that
\[
 d_V(t) \leq \underline{\alpha}N^2\int_{d_X(t)}^{d_X(0)}\psi(s)\,ds + \underline{\alpha}N^2\int_{d_X(0)}^{d_*}\psi(s)\,ds 
 \leq \underline{\alpha}N^2\int_{d_X(t)}^{d_*}\psi(s)\,ds.
\]
Since $d_V(t)\geq 0$, we have that the diameter in the position space is uniformly bounded. That is, $d_X(t)\leq d_*$ for all $t\geq 0 $. 
Defining $\psi_* = \psi(d_*)$, we note that $\psi(s) \geq \psi_*$ for $s \in [0,d_*]$. Using the inequality
\[
\frac{d}{dt}d_V(t) \leq -\underline{\alpha}N^2\psi(d_X(t))d_V + d_{\beta},
\]
we have that
\[
  \frac{d}{dt} d_V(t)\leq -\underline{\alpha}N^2\psi_*d_V + d_{\beta}.
\]
Hence
 \[ \label{d_V(t)-solution}
   d_V(t) \leq e^{-\underline{\alpha}N^2\psi_*t}d_V(0) + \int_0^t e^{-\underline{\alpha}N^2\psi_*(t-s)}d_{\beta}(s) ds.
 \]
Now let us show that the velocity diameter goes to zero asymptotically. The first term above goes to zero asymptotically in time. The second term can be written 
as 
\[
\frac{\int_0^t e^{\underline{\alpha}N^2\psi_* s} \, d_\beta(s) \, ds}{e^{\underline{\alpha}N^2\psi_* t}}. 
\]
There are two cases. If
\[
 \lim_{t \to \infty} \int_0^t e^{\underline{\alpha}N^2\psi_*s} \, d_{\beta}(s) \, ds < \infty,
\]
then this second term clearly limits to zero. On the other hand, the limit above is infinity and hence an application of L'Hospital's rule and the hypothesis that $\lim_{t\to \infty} d_{\beta}(t) = 0$ shows that
\[
\lim_{t \to \infty} \frac{\int_0^te^{\underline{\alpha}N^2\psi_*s}d_{\beta}(s)\,ds}{e^{\underline{\alpha}N^2\psi_*t}}
  =  \lim_{t \to \infty} \frac{e^{\underline{\alpha}N^2\psi_*t}d_{\beta}(t)}{e^{\underline{\alpha}N^2\psi_*t}}
 = \lim_{t\to \infty} d_{\beta}(t) = 0.
\]
\end{proof}

\section{Study of fast flocking with slow steering via singular perturbation approach} \label{flock-steer-pertubation-approach}%\cite{kokotovic1999singular}

We consider the model given by \eqref{eq-flock-steer-closed} and investigate the scenario where flocking is much faster than steering. In the singular perturbation approach, we capture this by the introduction of a small parameter $\epsilon$. For simplicity, we ignore the orientation bias and assume that $a_{ij}=\phi_{ij}(x)$. 
This leads us to the family of equations
\begin{equation}\label{eq-closed-loop-epsilon}
   \begin{aligned}
   \dot{x}_i &= v_i, \\
   \dot{v}_i &= \frac{\alpha_i}{\epsilon} (\overline{v}_i-v_i) +\beta_i,\\
   \overline{v}_i &= \sum_{j=1}^N \phi_{ij}(x) \, v_j,\\
   \alpha_i &= \xi_i(\overline{v}_i-v_i)\quad \forall i =1,\dots, N.
   \end{aligned}
\end{equation} 
Here, $0 < \epsilon \ll 1$ is a parameter that allows the model to flock rapidly.

Let $x_i(t,\epsilon)$ and $v_i(t,\epsilon)$ for all $i = 1,\dots,N$ be the solution of our new model \eqref{eq-closed-loop-epsilon}. We expand these solutions and some related variables of the model in a power series in $\epsilon$:
 \begin{equation}\label{eq-expansion}
 \begin{aligned}
	  x_i(t,\epsilon) &= x_{i,0}(t) + \epsilon x_{i,1}(t) + \dots, \\
 	 v_i(t,\epsilon) &= v_{i,0}(t) + \epsilon v_{i,1}(t) + \dots, \\
  	\alpha_i(t,\epsilon) &= \alpha_{i,0}(t) + \epsilon\alpha_{i,1}(t) +\dots,\\
 	 \overline{v}_i(t,\epsilon) &= \overline{v}_{i,0}(t) + \epsilon\overline{v}_{i,1}(t) +\dots,\\
 	\beta_i(t,\epsilon) &= \beta_{i,0}(t) + \epsilon \beta_{i,1}(t) + \dots. 
\end{aligned}
\end{equation}

\subsection{Leading order behavior}
We shall use $x_0(t)$ to denote $$(x_{1,0}(t),\dots,x_{N,0}(t)),$$ 
and likewise $v_0(t)$ and $\beta_0(t)$. We are interested in characterizing the leading order terms $x_0(t)$ and $v_0(t)$. In what follows, we frequently omit showing the dependence on time for brevity.  Substituting the expansions \eqref{eq-expansion} into \eqref{eq-closed-loop-epsilon} we obtain
 \begin{equation} \label{expansion-x-v}
 \begin{aligned}
   \dot{x}_{i,0}+\epsilon \dot{x}_{i,1}+\dots & = v_{i,0} +\epsilon v_{i,1}+\dots,\\ 
   \dot{v}_{i,0}+\epsilon\dot{v}_{i,1}+\dots &=\frac{1}{\epsilon}(\alpha_{i,0}+\epsilon \alpha_{i,1}+\dots)((\overline{v}_{i,0} - v_{i,0}),\\
   &+\epsilon (\overline{v}_{i,1} - v_{i,1})+\dots)+\beta_{i,0}(t) + \epsilon \beta_{i,1}(t) + \dots. 
\end{aligned}
\end{equation}
Furthermore we obtain
\begin{equation}\label{eq-vbar0and1}
  \begin{aligned}
    \overline{v}_{i,0}&= \sum_{j=1}^N\phi_{ij}(x_0) \, v_{j,0},\\
    \overline{v}_{i,1} & = \sum_{j=1}^N\phi_{ij}(x_0)v_{j,1} 
    				 + \sum_{j=1}^N\left\{\sum_{l=1}^N\sum_{k=1}^d \frac{\partial\phi_{ij}}{\partial x_l^k}(x_0) \, x_{l,1}^k \right\} v_{j,0}.
    \end{aligned}				 
  \end{equation}
We note that $x^k_i$ and $v^k_i$ are the $k$th components of the $i$th agent's position and velocity. Also $x^k_{i,0}$ and $x^k_{i,1}$ denote the leading order and the next order terms of $x^k_i$ and likewise for $v^k_{i,0}$ and  $v^k_{i,1}$.     
Balancing the terms of order $\epsilon ^{-1}$ in \eqref{expansion-x-v}, we obtain that
\begin{equation}
  \alpha_{i,0}(t)(\overline{v}_{i,0}(t)-v_{i,0}(t)) =0. \label{expan-pert-epsilon-1}
 \end{equation}  
This means that $\alpha_{i,0} = 0$ or $ \overline{v}_{i,0} - v_{i,0}=0$. since $ \underline{\alpha} = \min{\alpha_i} > 0$, we have that $ \overline{v}_{i,0} = v_{i,0}$.  Therefore 
 \[
 \sum_{j=1}^N \phi_{ij}(x_0)v_{j,0} = v_{i,0},
 \]
 and hence
\[
\sum_{j=1}^N \phi_{ij}(x_0)v_{j,0}^k = v_{i,0}^k,
\] where we use the superscript to denote the $k$th component of the velocity.
 Fixing a component $1 \leq k \leq d$ and writing the previous equation for all agents, we obtain
\begin{equation}\label{eq-Pv-v}
 	P(t) \, v_{0}^k(t) = v_{0}^k(t),
\end{equation}
where the matrix $P$ is given by:
\begin{equation}\label{eq-P}
P = 
   \begin{bmatrix}
	\phi_{11}(x_0)  &\cdots &\phi_{1N}(x_0) \\
%	\phi_{21}&\phi_{22}&\cdots &\phi_{2N} \\
	\vdots  & \ddots & \vdots\\
	\phi_{N1}(x_0) &\cdots &\phi_{NN}(x_0)
  \end{bmatrix},
\end{equation}
  and  $v^k = (v_{1,0}^k, \dots,v_{N,0}^k) \quad \forall  k = 1,\dots,d$.  
Since $P_{ij}=\phi_{ij}>0$ and 
\[
\sum_{j=1}^N P_{ij} =1,
\]
the matrix $P$ is a stochastic matrix. 
Since $P_{ij}>0$ for all $i,j$, $P$ has eigenvector $e = (1,\dots,1)^t$ corresponding to the eigenvalue $1$ of multiplicity one. Thus for each $k= 1,\dots,d$, 
\eqref{eq-Pv-v} has a unique solution for $v^k$ which is a multiple of $e=(1,\dots,1)^t$. This shows 
that $v_{i,0}(t)$ are all equal for $i=1,\dots,N$,
indicating flocking. We shall denote this flocking velocity by $v^f(t)$. 
  
Balancing the terms of order $\epsilon^0$ in \eqref{expansion-x-v} gives the system
  \begin{equation}\label{eq-eps0}
   \begin{aligned}
 	\dot{x}_{i,0}(t)& = v_{i,0}(t), \\ 
     \dot{v}_{i,0}(t) &= \alpha_{i,1}(\overline{v}_{i,0}(t)-v_{i,0}(t))
     +\alpha_{i,0}(\overline{v}_{i,1}(t)-v_{i,1}(t))+ \beta_{i,0}(t). 
   \end{aligned}
 \end{equation}
Since $v_{i,0}=v^f$ for all $i$, it follows that $\overline{v}_{i,0}= v^f$ for all $i$, and hence, from \eqref{eq-vbar0and1} we obtain that
\[
    \overline{v}_{i,1} 
    	 =\sum_{j=1}^N\phi_{ij}(x_0)v_{j,1} + \sum_{j=1}^N\left\{\sum_{l=1}^N\sum_{k=1}^d \frac{\partial\phi_{ij}}{\partial x_l^k}(x_0) \, x_{l,1}^k \right\} \, v^f.
\]
We change the order of the summation in the second term  and use
  the condition $\sum_{j =1}^N \phi_{i,j}(x) = 1$ to obtain that
\[
  \sum_{j=1}^N\left\{\sum_{l=1}^N\sum_{k=1}^d \frac{\partial}{\partial x_l^k}\phi_{ij}(x_0) \, x_{l,1}^k \right\}v^f = 
  	  \sum_{l=1}^N\left\{\sum_{k=1}^N x^k_{l,1} \, \frac{\partial }{\partial x_l^k}\left(\sum_{j=1}^d \phi_{i,j}(x_0)\right)\right\}v^f =0.
  \]
Thus $$  \overline{v}_{i,1}  = \sum_{j=1}^N\phi_{ij}(x_0)v_{j,1}.$$ 
Substituting these results in equation \eqref{eq-eps0}, we have that for each $i$
 \begin{equation}\label{accelaration-aft-flock}
  \dot{v}^f= \alpha_{i,0}(\overline{v}_{i,1}-v_{i,1})+ \beta_{i,0}. 
\end{equation}
 
From the first equation of \eqref{eq-eps0} we have that
 $$\dot{x}_{i,0}= v_{i,0}=v^f.$$
This implies that for each $i$
\begin{equation} \label{conf-pert-approach}
  x_{i,0}(t) = x_{i,0}(0) + \int_0^t v^f(s)\, ds.
\end{equation}
Hence for all $i$ and $j$
 \begin{equation}
  x_{i,0}(t) - x_{j,0}(t) = x_{i,0}(0) - x_{j,0}(0). \label{fixed-conf}
 \end{equation}
 It follows from \eqref{fixed-conf} that the leading order relative positions of agents do not change with time. Hence by the shift invariance assumption on $\phi_{ij}$, it follows that $\phi_{ij}(x_0(t))$ is independent of $t$. We denote by $\overline{a}_{ij}$:
\[
\overline{a}_{ij} = \phi_{ij}(x_0) \quad \forall i,j.
\]
Since $\overline{v}_{i,0}=v_{i,0}$, it follows that $\alpha_{i,0}=\xi_i(\overline{v}_{i,0}-v_{i,0}) = \xi_i(0) >0$.
Hence, for each $i$,
\begin{equation}
 \dot{v}^f= \xi_i(0)\, \left(\sum_{j=1}^N \overline{a}_{ij} \, v_{j,1} - v_{i,1}\right)+\beta_{i,0}. \label{accelarationflock1}
\end{equation} 
Taking the $k$th component in equation \eqref{accelarationflock1} we have that 
 \begin{equation} \label{flock-equation}
 \dot{v}^{f,k}= \xi_i(0) \, \left(\sum_{j=1}^N\overline{a}_{ij}  v_{j,1}^k - v_{i,1}^k\right)+\beta_{i,0}^k,
 \end{equation}
for $k=1,\dots,d$. 
We define for $1 \leq i,j \leq N$
\[
\begin{aligned}
 q_{ij}&= \xi_i(0) \, \overline{a}_{ij} \quad \forall i\neq j,\\ 
 q_{ii} &= \xi_i(0) \, \overline{a}_{ii}- \xi_i(0).
 \end{aligned}
\]
 The matrix $Q = [q_{ij}]$ is a transition rate matrix of a continuous time Markov chain. 
Moreover, since $q_{ij} = \xi(0) \, \overline{a}_{ij}>0$  for all $i \neq j$, the matrix $Q$ corresponds to an ergodic Markov chain in continuous time. Thus there exists a unique vector $(\pi_i)_{i=1}^N $ such that
$\sum_{i=1}^N \pi_i = 1$ and 
\[
\sum_{i=1}^N \pi_i\,q_{ij} = 0.
\]
With the introduction of matrix $Q$,  \eqref{flock-equation} may be written as
  \[
  	\dot{v}^{f,k}=\sum_{j=1}^N q_{ij} v_{j,1}^k +\beta_{i,0}^k.
  \]
  Multiplying by $\pi_i$ and summing over $ i= 1\dots N$, and using properties of $q_{ij}$ and $\pi_i$ we obtain that
\begin{equation}
   \dot{v}^{f,k} = \sum_{i=1}^N \pi_i \beta_{i,0}^k.
\end{equation}
Hence the flocking velocity $v^f(t)$ evolves according to the equation
\begin{equation} \label{eq:v_infty_beta}
\dot{v}^f =  \sum_{i=1}^N \pi_i \beta_{i,0}. 
\end{equation}

In general, one may expect the steering terms $\beta_i$ to depend on $x_i,v_i$ and possible $t$, so that 
\begin{equation}\label{eq-beta-feedback}
\beta_i(t) = \eta_i(x_i(t),v_i(t),t)    
\end{equation}
where we suppose $\eta_i:\real^d \times \real^d \times [0,\infty) \to \real^d$
is $C^1$ in its arguments. Then, it follows that the evolution equation for $v^f$ is given by 
\begin{equation}
\dot{v}^f(t) = \sum_{j=1}^N \pi_i(x_0(t)) \eta_i(x_{0,i}(t),v^f(t),t),
\end{equation}
where $x_{i,0}(t)$ are given by
\begin{equation}
x_0(t) = x(0) + \int_0^t v^f(s) \, ds.    
\end{equation}
Here $x(0)=(x_1(0),\dots,x_N(0))$ is the initial position of the agents and we observe that $\pi_i(x_0(t))$ is constant in time, since $\phi_{ij}(x_0(t))$ is constant in time.
We may summarize the leading order time evolution by the system of ODEs
\begin{equation}\label{eq-fastslow-reduced}
\begin{aligned}
\dot{x}_0(t) &= v^f(t),\\
\dot{v}^f(t) &= \sum_{j=1}^N \pi_i(x_0(t)) \eta_i(x_{0,i}(t),v^f(t),t).
\end{aligned}    
\end{equation}
This is a $(N+1)d$ dimensional system and 
the leading order velocities are given by $v_{i,0}(t)=v^f(t)$. We observe that in order to obtain a unique solution, we need an initial condition for $v^f(0)$ which may not be the true initial velocities $v_i(0)$ of the agents. Intuitively, one expects a rapid initial transient layer during which flocking occurs and the agents reach the flocking velocity $v^f(0)$. 

In the next subsection, we scale time to investigate this transient layer.
 
 \subsection{Initial transient layer} \label{subsec:transient-layer}
  The given problem has initial condition, $x(0) = (x_1(0)\dots,x_N(0))$ and $v(0) = (v_1(0),\dots,v_N(0))$. 
  We zoom into the transient layer at $t =0$ by introducing the variable $\tau = t/\epsilon$. We define $X$ and $V$ by 
  \[
  X(\tau,\epsilon) = x(t,\epsilon)= x(\epsilon \tau,\epsilon)  \quad\text{and} \quad
    V(\tau,\epsilon) = v(t ,\epsilon)= v(\epsilon \tau,\epsilon) .
  \]
  Differentiating with respect to $\tau$, we have that
 \[
   \frac{1}{\epsilon} \frac{dX_i(\tau,\epsilon)}{d \tau} =\frac{dx_i(t,\epsilon)}{dt} 
 \]
  and 
 \[
   \frac{1}{\epsilon} \frac{dV_i(\tau,\epsilon)}{d \tau} = \frac{dv_i(t,\epsilon)}{dt}. 
 \]
 With the change of variable we have the following system of differential equations: 
 \begin{equation}\label{eq-transient}
   \begin{aligned}
   X_i' &= \epsilon V_i, \\
   V_i' &= \alpha_i(\overline{V}_i - V_i) + \epsilon \beta_i,
   \end{aligned}
 \end{equation}
 where the prime denotes differentiation with respect to $\tau$. 
 The initial conditions to impose are 
  \begin{equation}
   \begin{aligned}
   X_i(0)&= x_i(0),\\
   V_i(0) &= v_i(0).
   \end{aligned}
 \end{equation}
% Note that agents do not have the same initial velocity.
As before, we assume an $\epsilon$-expansion for $X_i$ and $V_i$ of the following form:
 \begin{equation}
   \begin{aligned}
   X_i(\tau,\epsilon) &= X_{i,0}(\tau,\epsilon) + \epsilon X_{i,1}(\tau, \epsilon) +\dots,\\
   V_i(\tau,\epsilon) &= V_{i,0}(\tau,\epsilon) + \epsilon V_{i,1}(\tau, \epsilon) +\dots. 
   \end{aligned}
 \end{equation}
Substituting this expansion in \eqref{eq-transient} we obtain
 \[
   \begin{aligned}
   X'_{i,0} + \epsilon X'_{i,1}+\dots &= \epsilon( V_{i,0} +\epsilon V_{i,1} + \dots),\\
   V'_{i,0} + \epsilon V'_{i,1}+\dots &= (\alpha_{i,0} + \epsilon\alpha_{i,1}+\dots)( (\overline{V}_{i,0}-V_{i,0}) +\epsilon (\overline{V}_{i,1} -V_{i,1}) \dots),\\
   &+ \epsilon(\beta_{i,0} + \epsilon \beta_{i,1} +\dots).
   \end{aligned}
 \]
 Balancing the $\epsilon^0$ terms, we find that 
 \begin{equation} \label{boundary-layer-eps0}
 \begin{aligned}
  X'_{i,0} &= 0,\\
  V'_{i,0} &= \alpha_{i,0}(\overline{V}_{i,0}-V_{i,0}).
 \end{aligned}
 \end{equation}
It follows that $X_{i,0}(\tau) = X_i(0) = x_i(0)$.
This means that during the initial transient the leading order positions do not change in time $\tau$.

The model \eqref{boundary-layer-eps0} is similar to \eqref{eq-flock-steer-closed} without the steering terms, except that the positions $X_{i,0}$ are constant. Hence the influence matrix $a_{ij} = \phi_{ij}(X_0)$ is constant and strictly positive. 
Defining
\[
  d_X(\tau) = \max_{i,j}\|X_{i,0}(\tau)- X_{j,0}(\tau)\|, \;\; d_V(\tau) = \max_{i,j}\|V_{i,0}(\tau)-V_{j,0}(\tau)\|,
\]
to be the diameters in the position and the velocity spaces respectively, we see that the assumptions of Lemma \eqref{lem-alpha-bar} and Theorem \eqref{thm-closed-flock} are satisfied since the diameter in the steering space is zero. Thus Theorem \eqref{thm-closed-flock} can be invoked to conclude that $d_V(\tau) \to 0$ as $\tau \to \infty$.

Now let us find $\lim_{\tau \to \infty} V_{i,0}(\tau)$. The second equation of \eqref{boundary-layer-eps0} is 
\[
 \begin{aligned}
  V'_{i,0} &= \alpha_{i,0}(\overline{V}_{i,0}-V_{i,0})
           =   \alpha_{i,0}\left(\sum_{j=1}^N \phi_{ij}(X_0)V_{j,0}-V_{i,0}\right) 
           = \sum_{j=1}^N q_{ij}V_{j,0}.
  \end{aligned}         
\]
Where $Q = (q_{ij})$ is the same matrix that we have used in \eqref{flock-equation}. Taking the $k$th components and letting $Z_i^k = V_{i,0}^k$ and $Z^k = (Z^k_1,\dots,Z^k_N)$ we have $Z^{'k} = Q \, Z^k$. That is
\[
  Z^{'
  k}_i = \sum_{j=1}^N q_{ij} Z^k_j.
\]
Multiplying by $\pi_i$ and sum it from 1 to $N$, we have
\[
  \sum_{i=1}^N \pi_i Z^{'k}_i = \sum_{i=1}^N\sum_{j=1}^N \pi_i q_{ij} Z^k_j = \sum_{j=1}^N\left(\sum_{i=1}^N \pi_i q_{ij}\right)Z^k_j = 0.
\]
This implies that for $t \geq 0$,
\begin{equation} \label{eq-pi-boundary}
\sum_{i=1}^N \pi_i Z^k_i(t) = \sum_{i=1}^N \pi_i Z^k_i(0).
\end{equation}
 However, all the eigenvalues of $Q$ except for one zero eigenvalue have negative real parts. Thus $Z^k(t) \to \overline{Z}^k$ where $\overline{Z}^k$ is a multiple of $(1,\dots,1)^t$. That is $\overline{Z}_k = c_k \, (1,\dots,1)^t$.
  To find $c_k$, we take limits in \eqref{eq-pi-boundary}:
  \[
   \lim_{\tau\to\infty}\sum_{i=1}^N \pi_i Z^k_i(t) = c_k = \sum_{i=1}^N \pi_i Z^k_i(0) = \sum_{i=1}^N \pi_i V_{i,0}^k(0).
  \]
Using the matching condition $v^f(0)= \lim_{\tau \to \infty} V_{i,0}(\tau)$, we
deduce that 
  \begin{equation} \label{eq:initial-flocking-velocity}
   v^f(0)=  \lim_{\tau \to \infty} V_{i,0}(\tau) = (c_1,\dots,c_d) = \left(\sum_{i=1}^N \pi_i V_{i,0}^1(0),\dots,\sum_{i=1}^N \pi_i V_{i,0}^d(0) \right).
  \end{equation}
  
\section{Numerical Examples} \label{Numerical_examples}
In this section, we present some numerical simulations to illustrate our theoretical analysis. We consider the collection of $N = 7$ agents in two dimensions. 
We shall choose the initial positions and initial velocities randomly
(i.i.d.\ uniformly distributed) inside square regions 
$[0,8] \times [0,8]$ in position and $[0,3] \times [0,3]$ in velocity spaces respectively.

We assume all agents wish to follow the same circular trajectory 
\[
y(t) = (100+10\,\sin(0.1t), \;\; 10+10 \cos\,(0.1t))^t
\]
in the position space. 
We assume each agent $i$ implements a feedback law for steering according to 
\[
\beta_i(t) = \gamma_1(\dot{y}(t) - v_i(t)) + \gamma_2(y(t) - x_i(t)),
\]
Where $\gamma_1$
and $\gamma_2$ are two parameters. We remark that our circular target
trajectory is not intended to capture the mill ring phenomenon. In fact, we
explore a situation where the agents flock and closely follow the circular
target trajectory. Since they flock, they cannot possibly be spread out
in position along the circle as this would imply different velocities. 

In our MATLAB simulations, we took
$\gamma_1 = 2$ and $\gamma_2 = 0.1$.  We computed the solutions of the full
model \eqref{eq-flock-steer-closed} for $\epsilon = 0.1$ , $\epsilon = 0.01$
and $\epsilon = 0.001$. 

Additionally, we also computed the solution of the reduced model 
\eqref{eq-fastslow-reduced} obtained via the singular perturbation theory. In order to compute the correct 
initial flocking velocity $v^f(0)$ to be used in conjunction with \eqref{eq-fastslow-reduced}, we use the equation \eqref{eq:initial-flocking-velocity}. 

Finally, we computed the leading order approximation and compared it to the simulation when $\epsilon = 0.1$, $\epsilon = 0.01$ and  $\epsilon = 0.001$.  
For all $\epsilon$ we used the same randomly chosen initial conditions which
we provide here. Initial positions were 
\[
\begin{aligned}
   x_1(0) &= (6.8897,7.1568)^t, \;\;  x_2(0) = (1.6819,4.4079)^t, \;\; x_3(0) = (4.0103,5.8168)^t,\\
   x_4(0) &= (6.3834,6.7922)^t, \;\; x_5(0) = (2.4842,6.1173)^t, \;\; x_6(0) = (5.8959,1.4635)^t,\\ x_7(0) &= (1.0710,4.4853)^t,
\end{aligned}
\] 
and the initial velocities were 
\[
\begin{aligned}
  v_1(0) &= (2.8792,1.0212)^t,\;\; v_2(0) = (1.7558,0.6714)^t,\;\; v_3(0) = (2.2538,0.7653)^t,\\ 
  v_4(0) &= (1.5179,2.0972)^t,\;\; v_5(0) = (2.6727,2.8779)^t,\;\; v_6(0) =(1.6416,0.4159)^t , \\ v_7(0) &= (0.4479,07725)^t.
\end{aligned}
\]
In these simulations, we have used the following functions: 
\[
 \begin{aligned}
  \phi_{ij}(x,u) &= \frac{\phi(r_{ij})}{\sum_k \phi(r_{ik})} \, \,\text{where,}\, \, r_{ij} = \|x_j - x_i\|,\\
  \phi(r) &= \frac{1}{(1 + r^2)^{0.3}},\\
  \alpha_i(t) &= \xi_i(u_i) = \frac{10}{(0.1 + \|u_i\|^2)^{0.5}}\, \,\text{where,}\, \, u_i = \overline{v}_i - v_i.\\
 \end{aligned}
\]
%---------------------------------------1

\begin{figure}[tbhp]
\centering 
\subfloat[Positions]{\label{fig:sub1-0.1}\includegraphics[width=0.47\linewidth]{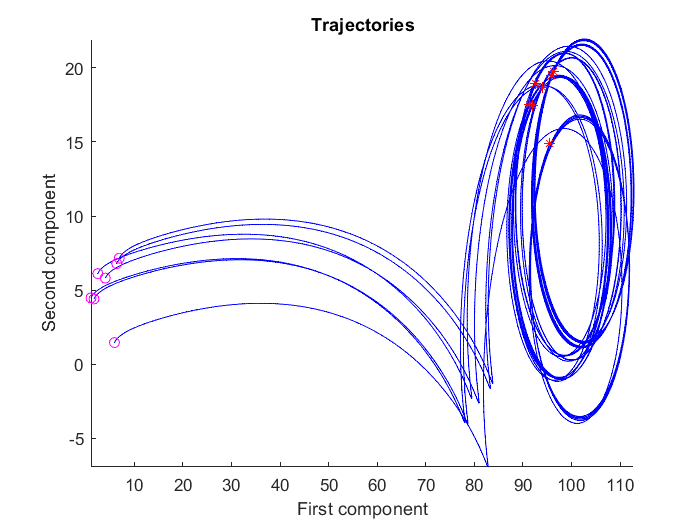}} 
\subfloat[Velocities]{\label{fig:sub2-0.1}\includegraphics[width=0.47\linewidth]{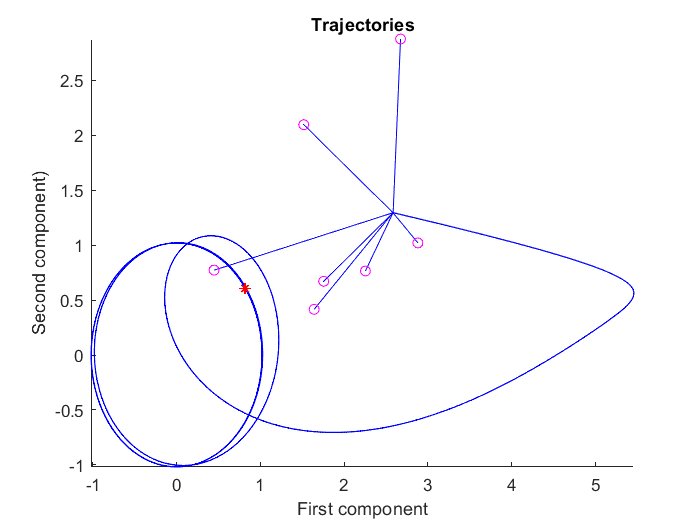}} 
\caption{Trajectories in position and velocity spaces. Magenta circles represent initial values and red stars the final values. Case $\epsilon = 0.1$.} 
\label{fig:trajectories-velocities-spaces-0.1} 
\end{figure}
%-----------------------------------------------------
\begin{figure}[tbhp]
\centering 
\subfloat[$x_i^1(t)$ and $y^1(t)$]{\label{fig:sub-x-component-0.1}\includegraphics[width=0.47\linewidth]{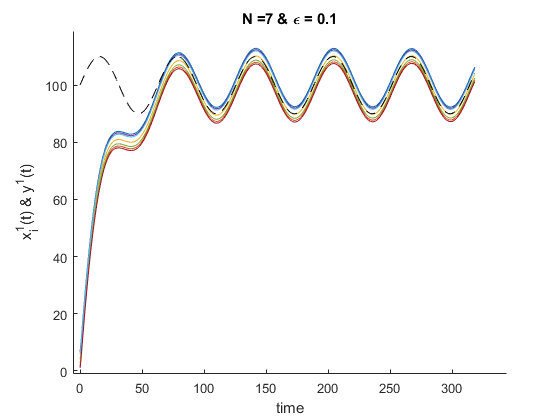}} 
\subfloat[$x_i^2(t)$ and $y^2(t)$]{\label{fig:sub-y-component-0.1}\includegraphics[width=0.47\linewidth]{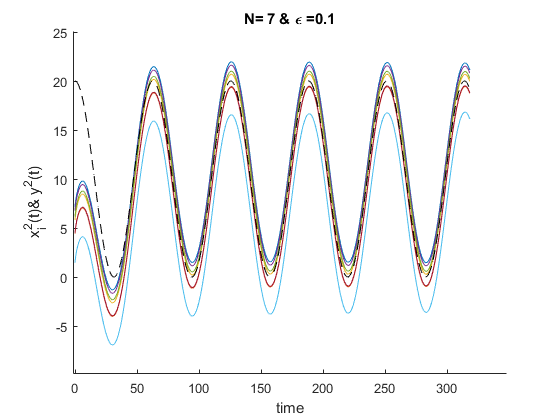}} 
\caption{Positions against time. Case $\epsilon = 0.1$. Target trajectory in dash black.} 
\label{fig:trajectories-components-0.1} 
\end{figure}
%------------------------------------------------------------3
\begin{figure}[tbhp]
\centering 
\subfloat[$v_i^1(t)$ and $\dot{y}^1(t)$ ]{\label{fig:sub-v1-component-0.1}\includegraphics[width=0.47\linewidth]{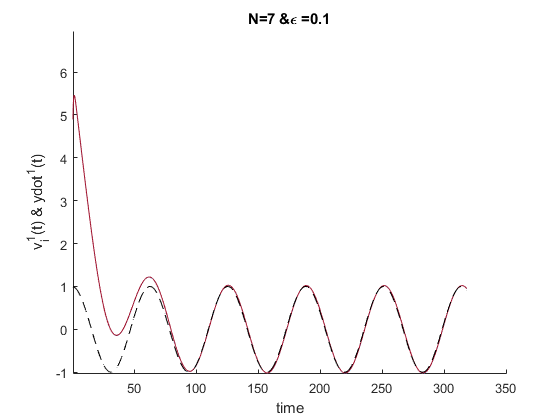}} 
\subfloat[$v_i^2(t)$ and $\dot{y}^2(t)$]{\label{fig:sub-v2-component-0.1}\includegraphics[width=0.47\linewidth]{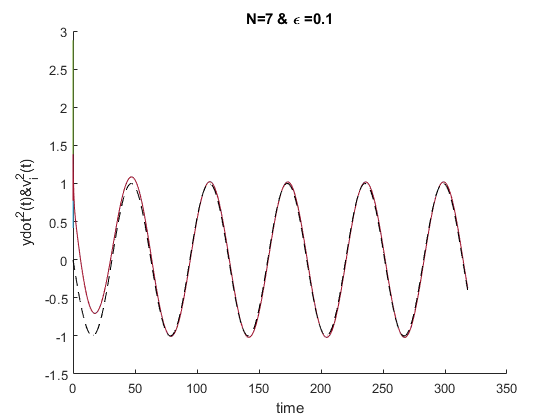}} 
\caption{Velocities against time. Case $\epsilon = 0.1$.Target velocity in dash black} 
\label{fig:velocities-components-0.1} 
\end{figure}

%------------------------------------------
\begin{figure}[tbhp]
\centering 
\subfloat[$v_i^1(t)$ and $\dot{y}^1(t)$]{\label{fig:sub-x-short-0.1}\includegraphics[width=0.47\linewidth]{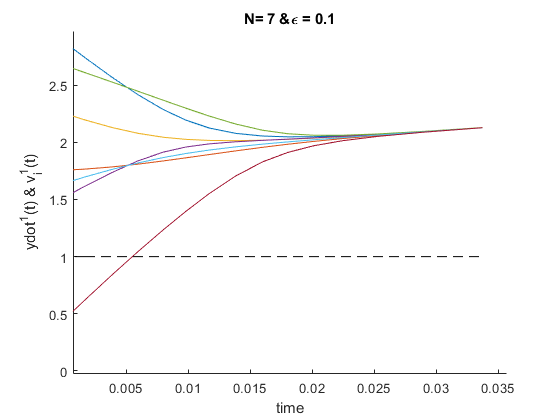}} 
\subfloat[$v_i^2(t)$ and $\dot{y}^1(t)$]{\label{fig:sub-y-short-0.1}\includegraphics[width=0.47\linewidth]{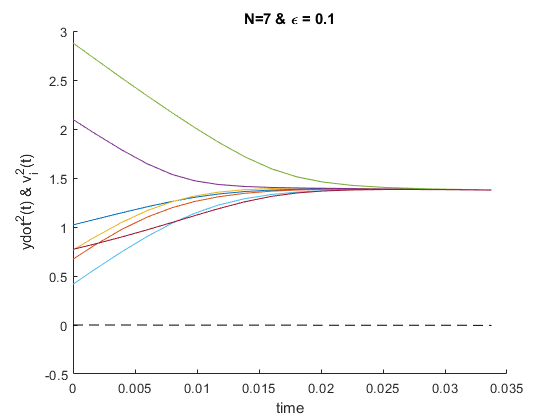}} 
\caption{Velocities against time for $t$ close to zero. Case $\epsilon = 0.1$ short representation. Target velocity in dash black}
\label{fig:boundary-layer-plotting-eps=0.1} 
\end{figure}
%-----------------------------------------------------------
%---------------------------------------1
\begin{figure}[tbhp]
\centering 
\subfloat[Positions]{\label{fig:target-error-sub1-0.01}\includegraphics[width=0.47\linewidth]{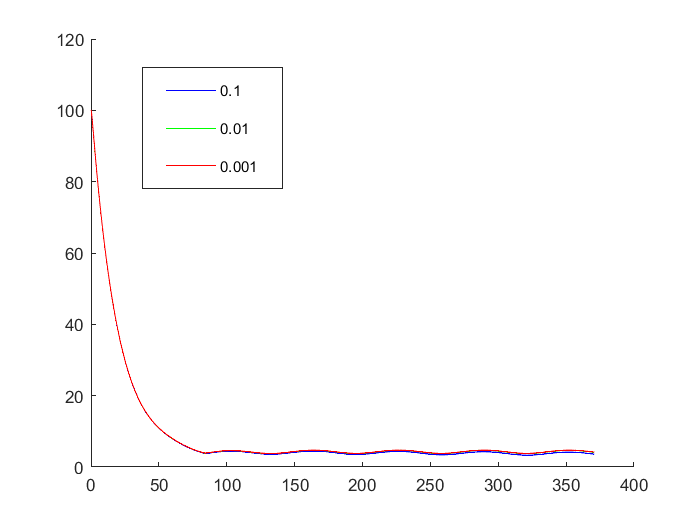}} 
\subfloat[Velocities]{\label{fig:target-error-sub2-0.01}\includegraphics[width=0.47\linewidth]{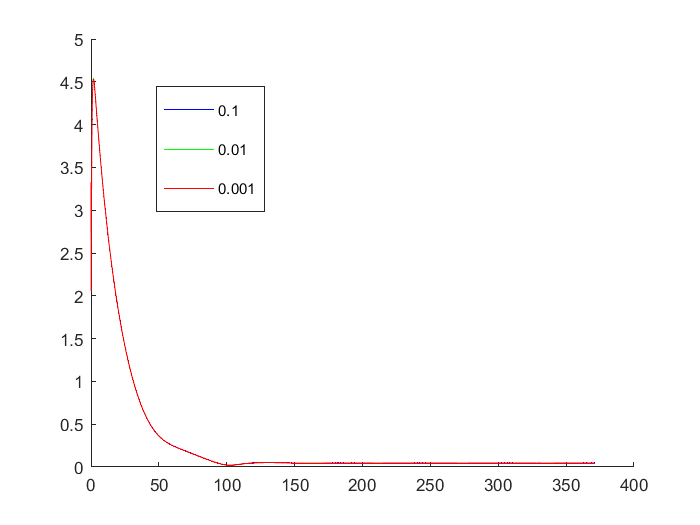}} 
\caption{Maximum target errors $\max_i\|x_i(t)-y(t)\|$ (position) and
  $\max_i\|v_i(t)-\dot{y}(t)\|$ (velocity) against $t$ for $\epsilon=0.1,0.01,0.001$. } 
\label{fig:target-error} 
\end{figure}
%-------------------------------------------
\newpage

 Figure \ref{fig:trajectories-velocities-spaces-0.1}
 shows the trajectories of all agents in the position and velocity spaces
 for the case of $\epsilon=0.1$.
 The small magenta circles represent
 initial values and the red stars represent the final time
 values. 
 Figure \ref{fig:sub2-0.1}
 shows that all agents flock to a common velocity and then stay
 together and steer towards the target velocity.
 Figure \ref{fig:trajectories-components-0.1}
 shows the plots of the components of trajectories in the position space against time. We see that
 after some time, all the components closely follow the components of the
 target trajectory (in dashed black). Similarly Figure
 \ref{fig:velocities-components-0.1} shows the components of the trajectories
 in the velocity space against time and both velocity alignment among agents
 and close tracking of the target are observed.
 The behavior of the system is similar for the other values of $\epsilon=0.01$
 and $\epsilon=0.001$ and are not shown. 
 Figure \ref{fig:boundary-layer-plotting-eps=0.1} clearly shows that
 the velocities of agents converge very fast to a common velocity  
and this common velocity is seen to evolve in Figure
\ref{fig:velocities-components-0.1}. Figure \ref{fig:target-error} shows
the maximum errors (measured in Euclidean norm) $\max_i \|x_i(t)-y(t)\|$ and $\max_i \|v_i(t)-\dot{y}(t)\|$
of positions and velocities with respect to the target, for the values of
$\epsilon=0.1, 0.01$ and $0.001$. The target tracking errors are
similar for the different $\epsilon$ values explored. We comment that the
feedback law we chose does not theoretically guarantee zero asymptotic
tracking error, but 
it is expected to track closely as observed in the figure.

\paragraph{Comparison of the leading order and the cases $\epsilon =0.1$, $\epsilon =0.01$ and $\epsilon =0.001$}\label{par:comparison_leading_order}

Figure \ref{fig:leading-order-errors} shows the numerically observed error
between the leading order approximation \eqref{eq-fastslow-reduced} and the
full model \eqref{eq-closed-loop-epsilon} for the values of $\epsilon=0.1,
0.01$ and $0.001$.
The errors shown are the maximum Euclidean norm errors
$\max_i \|x_i(t)-x_{i,0}(t)\|$ and $\max_i\|v_i(t)-v_{i,0}(t)\|$ of the positions
and the velocities as a function of time $t$. It is clear that as $\epsilon$ decreases,
the errors in the positions decrease as well. As for the velocities, the errors 
are quite small and decrease substantially from $\epsilon=0.1$ to $0.01$, but
do not show a big change from $\epsilon=0.01$ to $0.001$. We believe that
this is due to numerical errors in integrating the stiff system of ODEs.

\begin{figure}[tbhp]
\centering 
\subfloat[Positions]{\label{fig:leading-error-position}\includegraphics[width=0.47\linewidth]{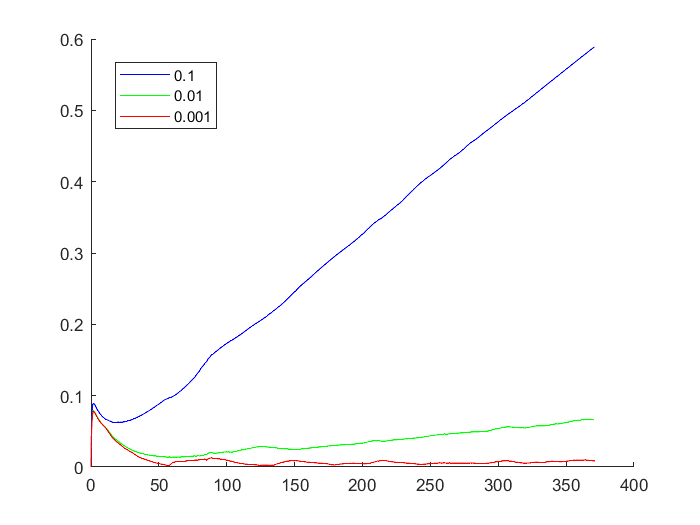}} 
\subfloat[Velocities]{\label{fig:leading-error-velocity}\includegraphics[width=0.47\linewidth]{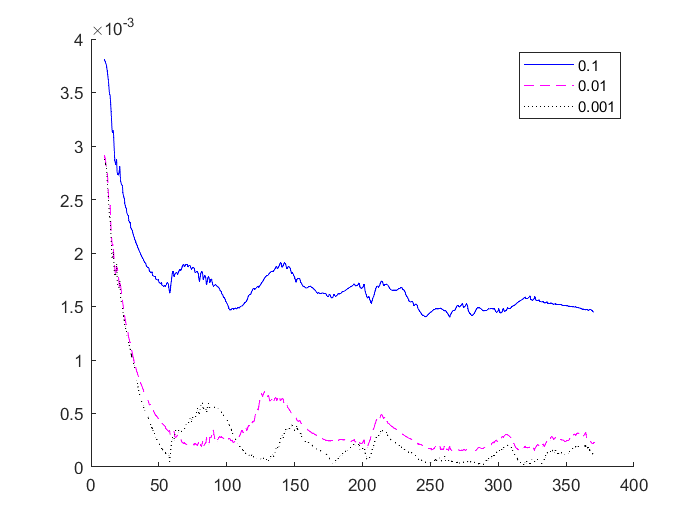}} 
\caption{Maximum errors $\max_i\|x_i(t)-x_{i,0}(t)\|$ and
  $\max_i\|v_i(t)-v_{i,0}(t)\|$ in the position and velocity spaces between
  the full model and the leading order approximation for $\epsilon=0.1, 0.01$
  and $0.001$. The error in velocity is shown starting at time $t=10$ as the
error is expected to be large during the initial transient. The errors in
position must be compared with the radius $10$ of the circle described in
position space and the errors in velocity must be compared with the radius $1$
of the circle described in the velocity space.} 
\label{fig:leading-order-errors} 
\end{figure}
 
\newpage
\section{Concluding remarks} 
We introduced and analyzed a generalized model of flocking with steering. In
our model, the acceleration of each agent has flocking and steering
components. The flocking component is a generalization of many existing models
and takes into account real world factors such as apriori bound on
acceleration, masking effects and orientation bias.   
We proved that the generalized model with steering flocks under certain sufficient conditions which naturally include assumptions on the steering components $\beta_i(t)$ of the accelerations of the agents. 
We also studied the case where flocking is much faster than steering using formal singular perturbation theory and showed that the leading order behavior is one where the agents flock  together with velocity $v^f(t)$ which evolves in time, see \ref{eq:v_infty_beta}. 
Our simulations showed that the leading order approximation was very similar to the real solution for small values of $\epsilon$ a scale parameter indicating the magnitude difference between flocking and steering accelerations. While this supports our formal derivation via singular perturbation theory, in future we would like to derive rigorous results that support the formal theory.

We also observe that the influence functions $\phi_{ij}$ were assumed to be nonvanishing for all $i,j \in {1,\dots, N}$ in our flocking results. This implies that the communication graph formed by the agents is strongly connected. In the case of the robotic systems, this will be computationally expensive. Even in the case of biological agents, all to all communication among agents may not be a reasonable assumption. This raises the question whether one could relax the strict positivity condition and still obtain flocking results. 

Our flocking results assumed that the steering components $\beta_i(t)$ of the agents were asymptotically in agreement ($\beta_i(t)-\beta_j(t) \to 0$ as $t \to \infty$). A related natural question is if the agents form subgroups within which this condition holds but fails across these subgroups, then can we obtain clusters of agents such that agents within each cluster flock together.  
%That raises the following question. can we ameliorate the model to achieve the same goal while some $\phi_{ij} = 0 ?$ The other concern is, how can we obtained multiples clusters that may move toward different directions?  Our future objective will be to answer these questions. We end our work by providing some numerical simulations that confirm our analysis. 
%\newpage

We showed that our model without the steering terms cannot exhibit the mill ring phenomenon
frequently observed in nature. The concept of flocking as defined by
researchers in the field involves zero (or asymptotically zero) velocity
diameter, and as such, it precludes the most interesting mill ring phenomenon
where the agents are spread out along the circle. While our model with
steering can exhibit such interesting mill ring phenomenon for appropriately chosen
steering functions, we did not explore this in this paper and is the subject
of future research. 

\appendix
\section{Useful lemmas} 
\begin{lemma}\label{lem-loclip-abscont}
Let $F : \real^n \to \real$ be locally Lipschitz and $u : [0,T] \to \real$ be absolutely continuous.
Then $F \circ u : [0,T] \to \real$ 
is absolutely continuous. 
\end{lemma}

\begin{lemma}\label{lem-max-deriv}
The function $f_i: \real \to \real$ be absolutely continuous on $[0,T]$ for $i=1,\dots,n$ and let $f:\real \to \real$ be defined by 
\[
f(t) = max\{ f_i(t) \, | \,  i = 1\dots n \}.
\]
Suppose $i_*:\real \to \{1,\dots,n\}$ satisfies $f_{i_*(t)}(t)  \geq f_j(t)$ for all $t$ and $j=1,\dots,n$.
Then $f$ is absolutely continuous and $f'(t) = f'_{i_*(t)}(t)$ for almost all $t$. 
\end{lemma} 

\begin{lemma}\label{lem-max-interval}
The forward maximal interval of existence of the model \ref{eq-flock-steer-open} is $[0, \infty)$ where we assume that $\alpha_i,\beta_i,a_{ij}$ are all continuous functions on $[0,\infty)$. % and \ref{eq-flock-steer-closed} 
%whenever there exist $M_{\beta} >0$ such that 
%$$ \|\beta_i\| \leq M_{\beta} \, \, \, \forall i .$$
\end{lemma}
\begin{proof}
   Let us suppose that the forward maximal interval of existence is the interval $[0,T^*)$, with $T^* < \infty$. We define the energy of the system by $E(t) = \max_i E_i(t) =  \max_i \frac{1}{2}\|v_i(t)\|^2.$
   Then
   \[
   \begin{aligned}
      \frac{dE(t)}{dt} &= \langle v_i,\dot{v}_i \rangle = \langle v_i , \alpha_i (\overline{v}_i-v_i) + \beta_i\rangle \\
      & = \alpha_i\langle v_i, \overline{v}_i\rangle - \alpha_i \langle v_i,v_i \rangle  + \alpha_i \langle v_i,\beta_i \rangle \\
      &\leq  \alpha_i\|v_i\|\left(\sum_j\, a_{ij}\|v_j\|\right)  - \alpha_i\|v_i\|^2 + \alpha_i\|v_i\|\|\beta_i\| \\
      &\leq \alpha_i\|v_i\|\|\beta_i\| 
  \end{aligned}
  \] 
   We have used the Cauchy-Schwartz inequality  and the conditions  $\|v_j\| \leq \|v_i\| $ and $\sum_j a_{ij} = 1$. We rearrange this inequality to
   \[
      \frac{dE(t)}{dt}  \leq \alpha_i (2E(t))^{\frac{1}{2}}\|\beta_i\| \leq \overline{\alpha}(2E(t))^{\frac{1}{2}}\|\beta_i\|.
   \] 
Where $\overline{\alpha}$ is the maximum of  $\alpha_i(t)$ over $i$ and $t \in [0,T^*]$. Multiplying this inequality by $(2E(t))^{-\frac{1}{2}}$ we may obtain
\[
%  \begin{aligned}
%   (2E(t))^{-\frac{1}{2}}\frac{dE(t)}{dt}  &\leq \overline{\alpha}\|\beta_i\| \\
   \frac{dE^\frac{1}{2}(t)}{dt}  \leq 2^\frac{1}{2}\overline{\alpha}\|\beta_i\|.
%   \int_0^t\frac{dE^\frac{1}{2}(\tau)}{d\tau} d\tau &\leq \int_0^t2^\frac{1}{2}\overline{\alpha}\|\beta_i(\tau)\|d\tau \leq \int_0^{T^*}2^\frac{1}{2}\overline{\alpha}\|\beta_i(\tau)\|d\tau 
%  \end{aligned}
\] 
%The steering vectors $\beta_i(t)$ are continuous by assumption \ref{ass-beta}, therefore, there exist a constant
Let $M_{\beta} > 0$ satisfy $\|\beta_i(t) \| \leq M_{\beta}$ for all $i$ and $t \in [0,T^*]$. We obtain
 \[
   \left( E^{\frac{1}{2}}(t) - E^{\frac{1}{2}}(0)\right) \leq 2^{-\frac{1}{2}}\overline{\alpha} M_{\beta}T^*.
 \]

And we deduce that 
  \[
    \|v_i\| \leq \left( E^{\frac{1}{2}}(0) + 2^{-\frac{1}{2}}\overline{\alpha} M_{\beta}T^*\right) < \infty.
  \]

We then deduce the upper bound of the vector position $x_i(t)$ as
\[
    \|x_i(t)\|\leq \left( E^{\frac{1}{2}}(0) + 2^{-\frac{1}{2}}\overline{\alpha} M_{\beta}T^*\right)T^* + \|x_i(0)\|.
\] 
Since the solution remains in a compact set for $t \in [0,T^*)$ we obtain a contradiction.  
\end{proof}

\section{Masking Effect}\label{sec:mask-example}
We provide an example of influence function with masking effect. Recall that the Cucker Smale influence function is given by  
 \[ 
 \phi^{CS}(r_{ij} )= \frac{1}{(1 + r_{ij}^2)^{\beta}}.
 \]
Here $r_{ij} = \|x_i - x_j\|$ is the distance between agents $i$ and agent $j$, and $\beta > 0 $ is a parameter of the model.
Before we define the influence function that takes into account the masking
effect, let us first define quantities $w_{ijl}$ for any three distinct agents
$i,j$ and $l$ by 
\[
  w_{ijl} =\frac{\langle x_j-x_i,x_l-x_i\rangle}{\sqrt{\|x_j-x_i\|^2\|x_l-x_i\|^2 + r^4}}, \,\,\,\,\,1\leq i,j,l\leq N,
\]
where $r>0$ is a fixed constant. We note that $w_{ijl}$ is a smoothed measure
of $\cos(\theta)$ where $\theta$ is the angle at the vertex $x_i$ of the
triangle formed by $x_i, x_j$ and $x_l$. 

  Agent $l$ masks agent $j$ from agent $i$, if and only if, the angles
  at both the vertices $x_i$ and $x_j$ (of the triangle formed by $x_i, x_j$
  and $x_l$) are close zero. Thus, if and only if
  $w_{ijl}+w_{jil} \approx 2$. Motivated by this observation, 
we define : 
\begin{equation} \label{eq:symmetric-inflce-with-ME}
\phi_{ij}^{ME}(x)  =\phi^{CS}(r_{ij})\,\exp \left\{-\sum_{\substack{1\leq l\leq N\\l\neq i \\l\neq j}} h(w_{ijl} + w_{jil})\right\},
\end{equation}
where, $h$ is defined below:
\[
\begin{aligned}
  h(s) &= 0 \,\,\,\,\,\ \,\,\,\,\,\,\,\,\,\,\text{ if } s \leq 1.9 ,\\
  h(s) & = 3000(s-1.9)^3 \,\,\,\text{ if } 1.9\leq s \leq 2.
\end{aligned}
\]

The exponential factor weighs the contribution to masking effect by all third agents
$l$ and 
reduces the influence of agent $j$ on agent $i$ accordingly.
Finally, we define influence function with masking effect by normalization:

\[
\phi_{ij}(x) = \frac{\phi_{ij}^{ME}(x)}{\sum_{k=1}^N \phi_{ik}^{ME}(x)}.
\]

\printbibliography

@article{ahn2012collision,
  title={On collision-avoiding initial configurations to Cucker-Smale type flocking models},
  author={Ahn, Shin Mi and Choi, Heesun and Ha, Seung-Yeal and Lee, Ho},
  journal={Communications in Mathematical Sciences},
  volume={10},
  number={2},
  pages={625--643},
  year={2012},
  publisher={International Press of Boston}
}

@article{ahn2010stochastic,
  title={Stochastic flocking dynamics of the Cucker--Smale model with multiplicative white noises},
  author={Ahn, Shin Mi and Ha, Seung-Yeal},
  journal={Journal of Mathematical Physics},
  volume={51},
  number={10},
  pages={103301},
  year={2010},
  publisher={American Institute of Physics}
}

@article{albi2014stability,
  title={Stability analysis of flock and mill rings for second order models in swarming},
  author={Albi, Giacomo and Balagu{\'e}, D and Carrillo, Jos{\'e} A and von Brecht, J32150701305},
  journal={SIAM Journal on Applied Mathematics},
  volume={74},
  number={3},
  pages={794--818},
  year={2014},
  publisher={SIAM}
}

@article{albi2016selective,
  title={Selective model-predictive control for flocking systems},
  author={Albi, Giacomo and Pareschi, Lorenzo},
  journal={arXiv preprint arXiv:1603.05012},
  year={2016}
}

@article{aureli2010coordination,
  title={Coordination of self-propelled particles through external leadership},
  author={Aureli, Matteo and Porfiri, Maurizio},
  journal={EPL (Europhysics Letters)},
  volume={92},
  number={4},
  pages={40004},
  year={2010},
  publisher={IOP Publishing}
}

@article{bongini2014emergence,
  title={emergence under perturbed and decentralized feedback controls},
  author={Bongini, Mattia and Fornasier, Massimo and Kalise, Dante},
  year={2014}
}

@article{canizo2011well,
  title={A well-posedness theory in measures for some kinetic models of collective motion},
  author={Canizo, Jos{\'e} A and Carrillo, Jos{\'e} A and Rosado, Jes{\'u}s},
  journal={Mathematical Models and Methods in Applied Sciences},
  volume={21},
  number={03},
  pages={515--539},
  year={2011},
  publisher={World Scientific}
}

@article{caponigro2015sparse,
  title={Sparse stabilization and control of alignment models},
  author={Caponigro, Marco and Fornasier, Massimo and Piccoli, Benedetto and Tr{\'e}lat, Emmanuel},
  journal={Mathematical Models and Methods in Applied Sciences},
  volume={25},
  number={03},
  pages={521--564},
  year={2015},
  publisher={World Scientific}
}

@article{carrillo2014nonlinear,
  title={Nonlinear stability of flock solutions in second-order swarming models},
  author={Carrillo, Jos{\'e} A and Huang, Yanghong and Martin, Stephan},
  journal={Nonlinear Analysis: Real World Applications},
  volume={17},
  pages={332--343},
  year={2014},
  publisher={Elsevier}
}

@article{carrillo2010self,
  title={Self-propelled interacting particle systems with roosting force},
  author={Carrillo, Jos{\'e} A and Klar, Axel and Martin, Stephan and Tiwari, Sudarshan},
  journal={Mathematical Models and Methods in Applied Sciences},
  volume={20},
  number={supp01},
  pages={1533--1552},
  year={2010},
  publisher={World Scientific}
}

@article{cucker2010avoiding,
  title={Avoiding collisions in flocks},
  author={Cucker, Felipe and Dong, Jiu-Gang},
  journal={IEEE Transactions on Automatic Control},
  volume={55},
  number={5},
  pages={1238--1243},
  year={2010},
  publisher={IEEE}
}

@article{cucker2007emergent,
  title={Emergent behavior in flocks},
  author={Cucker, Felipe and Smale, Steve and others},
  journal={IEEE Transactions on automatic control},
  volume={52},
  number={5},
  pages={852--862},
  year={2007},
  publisher={New York, NY: Institute of Electrical and Electronics Engineers, 1963-}
}

@article{cucker2004modeling,
  title={Modeling language evolution},
  author={Cucker, Felipe and Smale, Steve and Zhou, Ding-Xuan},
  journal={Foundations of Computational Mathematics},
  volume={4},
  number={3},
  pages={315--343},
  year={2004},
  publisher={Springer}
}

@article{d2006self,
  title={Self-propelled particles with soft-core interactions: patterns, stability, and collapse},
  author={D’Orsogna, Maria R and Chuang, Yao-Li and Bertozzi, Andrea L and Chayes, Lincoln S},
  journal={Physical review letters},
  volume={96},
  number={10},
  pages={104302},
  year={2006},
  publisher={APS}
}

@article{ha2010asymptotic,
  title={Asymptotic dynamics for the Cucker--Smale-type model with the Rayleigh friction},
  author={Ha, Seung-Yeal and Ha, Taeyoung and Kim, Jong-Ho},
  journal={Journal of Physics A: Mathematical and Theoretical},
  volume={43},
  number={31},
  pages={315201},
  year={2010},
  publisher={IOP Publishing}
}

@article{ha2010emergent,
  title={Emergent behavior of a Cucker-Smale type particle model with nonlinear velocity couplings},
  author={Ha, Seung-Yeal and Ha, Taeyoung and Kim, Jong-Ho},
  journal={IEEE Transactions on Automatic Control},
  volume={55},
  number={7},
  pages={1679--1683},
  year={2010},
  publisher={IEEE}
}

@article{ha2009emergence,
  title={Emergence of time-asymptotic flocking in a stochastic Cucker-Smale system},
  author={Ha, Seung-Yeal and Lee, Kiseop and Levy, Doron},
  journal={Communications in Mathematical Sciences},
  volume={7},
  number={2},
  pages={453--469},
  year={2009},
  publisher={International Press of Boston}
}

@article{ha2009simple,
  title={A simple proof of the Cucker-Smale flocking dynamics and mean-field limit},
  author={Ha, Seung-Yeal and Liu, Jian-Guo},
  journal={Communications in Mathematical Sciences},
  volume={7},
  number={2},
  pages={297--325},
  year={2009},
  publisher={International Press of Boston}
}

@article{haskovec2013flocking,
  title={Flocking dynamics and mean-field limit in the Cucker--Smale-type model with topological interactions},
  author={Haskovec, Jan},
  journal={Physica D: Nonlinear Phenomena},
  volume={261},
  pages={42--51},
  year={2013},
  publisher={Elsevier}
}

@article{hendrickx2013convergence,
  title={Convergence of type-symmetric and cut-balanced consensus seeking systems},
  author={Hendrickx, Julien M and Tsitsiklis, John N},
  journal={IEEE Transactions on Automatic Control},
  volume={58},
  number={1},
  pages={214--218},
  year={2013},
  publisher={IEEE}
}

@article{levine2011necessary,
  title={On necessary and sufficient conditions for differential flatness},
  author={L{\'e}vine, Jean},
  journal={Applicable Algebra in Engineering, Communication and Computing},
  volume={22},
  number={1},
  pages={47--90},
  year={2011},
  publisher={Springer}
}

@article{motsch2011new,
  title={A new model for self-organized dynamics and its flocking behavior},
  author={Motsch, Sebastien and Tadmor, Eitan},
  journal={Journal of Statistical Physics},
  volume={144},
  number={5},
  pages={923},
  year={2011},
  publisher={Springer}
}

@article{olfati2006flocking,
  title={Flocking for multi-agent dynamic systems: Algorithms and theory},
  author={Olfati-Saber, Reza},
  journal={IEEE Transactions on automatic control},
  volume={51},
  number={3},
  pages={401--420},
  year={2006},
  publisher={IEEE}
}

@article{park2010cucker,
  title={Cucker-Smale flocking with inter-particle bonding forces},
  author={Park, Jaemann and Kim, H Jin and Ha, Seung-Yeal},
  journal={IEEE Transactions on Automatic Control},
  volume={55},
  number={11},
  pages={2617--2623},
  year={2010},
  publisher={IEEE}
}

@article{shen2008cucker,
  title={Cucker--Smale flocking under hierarchical leadership},
  author={Shen, Jackie},
  journal={SIAM Journal on Applied Mathematics},
  volume={68},
  number={3},
  pages={694--719},
  year={2008},
  publisher={SIAM}
}

@article{stamoulas2018convergence,
  title={Convergence, stability, and robustness of multidimensional opinion dynamics in continuous time},
  author={Stamoulas, Serap Tay and Rathinam, Muruhan},
  journal={SIAM Journal on Control and Optimization},
  volume={56},
  number={3},
  pages={1938--1967},
  year={2018},
  publisher={SIAM}
}

@article{van1998differential,
  title={Differential flatness and absolute equivalence of nonlinear control systems},
  author={van Nieuwstadt, Michiel and Rathinam, Muruhan and Murray, Richard M},
  journal={SIAM Journal on Control and Optimization},
  volume={36},
  number={4},
  pages={1225--1239},
  year={1998},
  publisher={SIAM}
}

@article{vicsek1995novel,
  title={Novel type of phase transition in a system of self-driven particles},
  author={Vicsek, Tam{\'a}s and Czir{\'o}k, Andr{\'a}s and Ben-Jacob, Eshel and Cohen, Inon and Shochet, Ofer},
  journal={Physical review letters},
  volume={75},
  number={6},
  pages={1226},
  year={1995},
  publisher={APS}
  }
\end{document}